\documentclass{amsart}
\usepackage[T1]{fontenc}
\usepackage[utf8]{inputenc}
\usepackage{lmodern}
\usepackage[a4paper,margin=31mm]{geometry}
\usepackage{amsmath,amssymb,amsfonts,amsthm}
\usepackage{xcolor}
\usepackage{hyperref}
\definecolor{linkblue}{RGB}{36,73,118}
\definecolor{citegreen}{RGB}{28,100,70}
\hypersetup{
  colorlinks=true,
  linkcolor=linkblue,
  citecolor=citegreen,
  urlcolor=linkblue,
  pdftitle={Compact LCK Manifolds and Balanced Threefolds with Constant Holomorphic Sectional Curvature},
  pdfauthor={Junpeng Li}
}
\allowdisplaybreaks[2]
\numberwithin{equation}{section}

\theoremstyle{plain}
\newtheorem{theorem}{Theorem}[section]
\newtheorem{proposition}[theorem]{Proposition}
\newtheorem{lemma}[theorem]{Lemma}
\newtheorem{corollary}[theorem]{Corollary}
\newtheorem{conjecture}{Conjecture}
\theoremstyle{definition}
\newtheorem{definition}[theorem]{Definition}
\theoremstyle{remark}

\newcommand{\C}{\mathbb C}
\newcommand{\R}{\mathbb R}
\newcommand{\Ric}{\operatorname{Ric}}
\newcommand{\trg}{\operatorname{tr}_{g}}
\newcommand{\Boch}{\operatorname{Boch}}
\newcommand{\Deck}{\operatorname{Deck}}
\newcommand{\ddbar}{\sqrt{-1}\,\partial\bar\partial}
\newcommand{\ii}{\sqrt{-1}}
\newcommand{\DivCh}{\operatorname{div}^{1,0}}
\newcommand{\rhoone}{\rho^{(1)}}
\newcommand{\rhotwo}{\rho^{(2)}}
\newcommand{\rhothree}{\rho^{(3)}}
\newcommand{\norm}[1]{\lVert #1\rVert}
\newcommand{\ip}[2]{\left\langle #1,#2\right\rangle}
\newcommand{\dd}{\mathop{}\!\mathrm{d}}
\DeclareMathOperator{\tr}{tr}

\def\XXint#1#2#3{{\setbox0=\hbox{$#1{#2#3}{\int}$ }
\vcenter{\hbox{$#2#3$ }}\kern-.6\wd0}}

\title[Holomorphic sectional curvature on balanced and LCK threefolds]
  {compact balanced threefolds and LCK manifolds with constant holomorphic sectional curvature}
  
\author{Shuwen Chen}
\address{Shuwen Chen. School of Mathematical Sciences, Chongqing Normal University, Chongqing 401331, China}
\email{3153017458@qq.com}

\author{Junpeng Li}
\address{Junpeng Li. School of Mathematical Sciences, Zhejiang Normal University,
Jinhua, Zhejiang 321004, China}
\email{junpengli@zjnu.edu.cn}\thanks{}

\date{}

\begin{document}

\begin{abstract}
A long-standing conjecture in Hermitian geometry says that a compact Hermitian manifold with constant Chern holomorphic sectional curvature $c$ is K\"ahler for $c\neq 0$ and Chern flat for $c=0$. Although the conjecture has been established in complex dimension two, it remains open in general in higher dimensions. We verify the conjecture for compact balanced threefolds when $c\leq 0$. For compact locally conformally K\"ahler manifolds, Chen, Chen, and Nie established the case $c\leq 0$, while Huang and Wan recently settled the remaining case. Inspired by the approach of Huang and Wan, we investigate a generalization of the conjecture for canonical metric connections and establish it for connected compact locally conformally K\"ahler manifolds.
\end{abstract}

\maketitle
\enlargethispage{2pt}

\section{Introduction}\label{sec:introduction}

The simplest kind of Riemannian manifolds are complete Riemannian manifolds with constant sectional curvature, known as {\em space forms}. Their universal covers are respectively the sphere $S^n$, the Euclidean space ${\mathbb R}^n$, or the hyperbolic space ${\mathbb H}^n$, equipped with (scaling of) the standard metrics.

In the Hermitian setting of complex dimension \(n\geq2\), it is no
longer possible to require the sectional curvature to be constant
unless the Hermitian metric is flat. Instead, one requires the {\em holomorphic sectional curvature} to be constant. Complete K\"ahler manifolds with constant holomorphic sectional curvature are called {\em complex space forms}. Analogous to the real case, they are quotients of complex projective space ${\mathbb C}{\mathbb P}^n$, the complex Euclidean space ${\mathbb C}^n$, or the complex hyperbolic space ${\mathbb C}{\mathbb H}^n$, equipped with (scaling of) the standard metrics.

Given a Hermitian manifold $(M^n,g)$, there are three canonical connections associated with $g$ that have been extensively studied: the Levi-Civita (Riemannian) connection $\nabla^{LC}$, the Chern connection $\nabla^c$, and the Bismut (also known as Strominger) connection $\nabla^{s}$. The last connection is compatible with both the metric $g$ and the almost complex structure $J$, and it also has totally skew-symmetric torsion. Its existence and uniqueness were proved by Bismut \cite{Bismut}, and it was independently discovered by Strominger \cite{Strominger}, hence, it is also called the Strominger connection in some literature.

We now recall the definition of holomorphic sectional curvature. Given a Hermitian manifold $(M^n,g)$, the holomorphic sectional curvature $H^c$ of $\nabla^c$ is defined by
\[
H^c(X)
=
\frac{R^c(X,\bar X,X,\bar X)}{|X|^4} ,\]
where $X\neq 0$ is any complex tangent vector of type $(1,0)$, and $R^c$ is the curvature tensor of $\nabla^c$. Similarly, let $R^{LC}$ and $R^{s}$ denote the curvature tensors of $\nabla^{LC}$ and $\nabla^{s}$, respectively. Then the corresponding holomorphic sectional curvatures $H^{LC}$ and $H^{s}$ can be similarly defined by replacing $R^c$ in the above formula with $R^{LC}$ or $R^{s}$, respectively. When $g$ is not K\"ahler, its curvature tensor generally does not satisfy all the K\"ahler symmetries, and hence $H^c$ does not determine the entire curvature tensor $R^c$. A natural question is what kinds of complete Hermitian manifolds can admit constant holomorphic sectional curvature when the K\"ahler condition is no longer assumed. A long-standing conjecture in Hermitian geometry is the following:

\begin{conjecture}[{\bf Constant Holomorphic Sectional Curvature Conjecture}] \label{conj1}
Let $(M^n,g)$ be a compact Hermitian manifold with $n\geq 2$. Assume that $H^c=c$ (or $H^{LC}=c$) is a constant. If $c\neq 0$, then $g$ is K\"ahler, and if $c=0$, then $g$ is Chern (or Levi-Civita) flat, namely $R^c=0$ (or $R^{LC}=0$).
\end{conjecture}

Note that the compactness assumption is necessary for the validity of the conjecture, without it there are counterexamples in the noncompact case. Compact Chern flat manifolds have been classified by Boothby \cite{Boothby} as all the compact quotients of complex Lie groups equipped with left invariant metrics. Moreover, compact Levi-Civita flat Hermitian threefolds have been classified by Khan-Yang-Zheng in \cite{KYZ}, but for $n\geq4$, the classification remains an open question.

For $n=2$, Conjecture \ref{conj1} is known to be true by the combined effort of Balas-Gauduchon \cite{Balas, BG} in 1985, Sato-Sekigawa \cite{SS} in 1990, and Apostolov-Davidov-Muskarov \cite{ADM} in 1996.

For $n\geq 3$, the first substantial result concerning this conjecture was obtained by  Davidov-Grantcharov-Muskarov \cite{DGM}, who showed that the only twistor space with constant holomorphic sectional curvature is the complex space form ${\mathbb C}{\mathbb P}^3$. More recently, for the Chern version of the conjecture, Chen-Chen-Nie \cite{ChenChenNie} confirmed the conjecture for locally conformally K\"ahler manifolds with $c\leq 0$. Tang in \cite{Tang} proved the conjecture under the additional assumption that the metric is {\em Chern K\"ahler-like} (namely, $R^c$ obeys all the K\"ahler symmetries). Zhou-Zheng in \cite{ZhouZheng} proved the conjecture under the additional assumption that the manifold is a compact balanced threefold with zero {\em real bisectional curvature}, a notion introduced in \cite{YangZheng} that is slightly stronger than $ H^c $, showing that such a manifold must be Chern flat.  Ma--Nie in \cite{MaNie} verified the conjecture for compact normal balanced threefolds under the additional assumption that the constant
holomorphic sectional curvature satisfies $c\leq 0$. Recently, Wang-Zheng in \cite{WangZheng} verified the conjecture for all compact balanced Bismut torsion-parallel fourfolds. Furthermore, Li-Zheng in \cite{LiZheng} or Rao-Zheng in \cite{RaoZheng} confirmed the conjecture for $ \nabla^c $ under the additional assumption that either $ (M^n, g) $ is a complex nilmanifold, or $g$ is {\em Bismut K\"ahler-like} (namely, $R^{s}$ obeys all K\"ahler symmetries).

We consider balanced manifolds. Recall that a Hermitian manifold $(M^n,J,g)$ is said to be balanced \cite{Michelsohn} if its fundamental form $\omega_g$ satisfies
\[
d\bigl(\omega_g^{n-1}\bigr)=0.
\]
Balanced manifolds form an important class of Hermitian manifolds
and include many non-K\"ahler examples even in complex dimension
three. For instance, this class includes all twistor spaces and many
known examples of non-K\"ahler Calabi--Yau manifolds. The conjecture
remains open for balanced manifolds in arbitrary dimensions. We
consider the case of complex dimension three.

We prove the nonpositive case.

\begin{theorem}\label{thm:balanced-main}
Let \((M^3,g)\) be a compact balanced threefold with
nonpositive constant Chern holomorphic sectional curvature \(c\).
If \(c=0\), then \(g\) is Chern flat.  If \(c<0\), then \(g\) is K\"ahler.
\end{theorem}

In the negative case, \(g\) is locally complex hyperbolic up to
scaling.  In the zero case, \(g\) need not be K\"ahler. For example, the standard invariant Hermitian metric on the Iwasawa
manifold is balanced and Chern flat \cite{OtalUgarte}.  The Iwasawa
manifold is not K\"ahler \cite{FuWangWu}.

On a Hermitian manifold $ (M^n,J, g) $, the three canonical connections, namely the Levi-Civita connection $\nabla^{LC}$, the Chern connection $\nabla^c$, and the Bismut connection $\nabla^{s}$, all equal to each other when $ g $ is K\"ahler; and they are mutually distinct when $g$ is not K\"ahler. The one-parameter family of Hermitian connections interpolating between $\nabla^c$ and $\nabla^{s}$, known as the {\em Gauduchon connections}, was introduced by Gauduchon in \cite{Gauduchon}:
\begin{equation}\label{eq:gauduchon-intro}
 \nabla^t
 =\frac{1+t}{2}\nabla^c+\frac{1-t}{2}\nabla^s
 =t\nabla^c+(1-t)\nabla^l,
 \qquad t\in\R,
\end{equation}
In particular, $\nabla^1=\nabla^c$ is the Chern connection, $\nabla^{\!-\!1}=\nabla^s$ is the Bismut connection, while $\nabla^0=\nabla^l=\frac{1}{2}(\nabla^c+\nabla^s)$ is the Hermitian projection of the Levi-Civita connection $\nabla^{LC}$, often called {\em Lichnerowicz connection}. Note that replacing the Chern connection in  Conjecture \ref{conj1} with other connections, such as the Levi-Civita connection, Bismut connection, or Gauduchon connection, gives rise to related conjectures. See \cite{ChenNie, ChenZheng, ChenZ1} and the references therein for further details.

Next, we may consider the two-parameter plane of metric connections spanned by $\nabla^{LC}$, $\nabla^c$ and $\nabla^{s}$ (see for instance \cite{ZhaoZheng}):
\begin{equation}\label{eq:two-parameter-intro}
 \mathcal D_s^t
 :=(1-s)\nabla^t+s\nabla^{\mathrm{LC}},
 \ \ \ \ \ \ \ \ \ \ (t,s) \in \Omega := \{ s\neq 1\} \cup \{ (0,1)\} \subset {\mathbb R}^2.
\end{equation}
In the $ts$-plane ${\mathbb R}^2$, the domain $\Omega$ can be viewed as the cone over the $t$-axis with vertex $(0,1)$, or equivalently, the {\em plane of canonical metric connections} $\mathcal D_s^t$ are the cone over the line of Gauduchon connections with vertex at the Levi-Civita connection. It is well-known that when the metric $g$ is K\"ahler, all $\mathcal D_s^t$ coincide. In contrast, if $g$ is non-K\"ahler, then $\mathcal D^t_s\neq \mathcal D^{t'}_{s'}$ for any two distinct points $(t,s)$, $(t',s') \in \Omega$. Chen--Nie \cite{ChenNie} obtained a characterization of the possible extension of Conjecture \ref{conj1} to the family $\mathcal D_s^t$. They discovered a distinguished subset (which will be called the {\em Chen-Nie curve} following \cite{ChenZ25})
\begin{equation*} \label{eq:Omega0}
\Gamma = \{ (t,s)\in {\mathbb R}^2 \mid (1-t+ts)^2 + s^2 = 4 \} \subset \Omega ,
\end{equation*}
and proposed the following conjecture:

\begin{conjecture}[{\bf Chen-Nie}]  \label{conj2}
Let $(M^n,g)$ be a compact Hermitian manifold with $n\geq 2$.. Suppose the holomorphic sectional curvature of the connection $\mathcal D_s^t$ is a nonzero constant. Then $g$ must be K\"ahler, and hence a complex space form. If the holomorphic sectional curvature of $\mathcal D_s^t$ vanishes and $(t,s)\in \Omega \setminus \Gamma$, then $g$ must be $\mathcal D_s^t$ flat.
\end{conjecture}
A companion question was raised in \cite{ChenZ25} as follows:

\vspace{0.3cm}

\noindent {\bf Question 3.}  {\em For any $(t,s)\in \Gamma$, what kind of compact Hermitian manifolds admit a $\mathcal D_s^t$ connection that is non-flat but has vanishing holomorphic sectional curvature?
}

\vspace{0.3cm}

For $n=2$, Conjecture $\ref{conj2}$ was proved by Chen-Nie in \cite{ChenNie}. For $n\geq 3$, as shown in \cite{ChenZ25}, the first named author and Zheng confirmed the conjecture for complex nilmanifolds with nilpotent $J$ in the sense of \cite{CFGU} and for non-balanced Bismut torsion-parallel manifolds. Recall that a compact Hermitian manifold $(M^n,J,g)$ is said to be {\em Bismut torsion-parallel} (or BTP for brevity) if the torsion $T^b$ of the Bismut connection $\nabla^{s}$ satisfies $\nabla^sT^b=0$. Examples and properties of BTP manifolds are discussed in \cite{ZhaoZ22, ZhaoZ24}. Note that any non-K\"ahler Bismut K\"ahler-like (BKL) manifold is always a non-balanced BTP manifold.

Locally conformally K\"ahler (LCK for short) manifolds form an
important class of Hermitian manifolds. A Hermitian manifold
\((M^n,J,h)\), \(n\geq2\), is called LCK if there exists a real
one-form \(\theta\) such that
\[
 d\omega_h=\theta\wedge\omega_h,
 \qquad
 d\theta=0.
\]
where \(\omega_h(X,Y)=h(JX,Y)\).  The one-form \(\theta\) is called the
Lee form \cite{DragomirOrnea}.  The metric \(h\) is called globally
conformally K\"ahler (GCK for short) if \(\theta\) is exact.  For further details on locally conformally K\"ahler manifolds, we refer the reader to the book by Ornea and Verbitsky \cite{OV} and the references therein. We consider compact LCK manifolds. Our main result is as follows.

\begin{theorem}\label{thm:main-classification}
Let \((M^n,h)\) be a connected compact LCK manifold with $n\geq 2$. Suppose that \(h\) has pointwise constant holomorphic sectional
curvature \(\kappa\) with respect to a Gauduchon connection
\(\nabla^t\). Then either\/ {\rm (i)} \(h\) is K\"ahler, or\/
{\rm (ii)} a holomorphic cover of \((M,h)\) is an isosceles Hopf
manifold with an admissible metric, and in this case \(t=-1\) or
\(t=3\). Moreover, if \(\kappa\) is globally constant \(c\), then in case\/ {\rm (ii)} the admissible metric
is the standard Hopf metric up to scaling and \(c=0\). In particular,
if \(t\notin\{-1,3\}\), then \(h\) is K\"ahler.
\end{theorem}

Isosceles Hopf manifolds and admissible metrics are defined in
Definition~\ref{def:admissible}.  Conversely, every admissible metric on
an isosceles Hopf manifold has pointwise constant Gauduchon holomorphic
sectional curvature for \(t=-1\) and \(t=3\) \cite{ChenNie}.

\begin{theorem}\label{thm:two-parameter-classification}
Let \((M^n,h)\) be a connected compact LCK manifold with $n\geq 2$. Suppose that \(h\) has pointwise constant holomorphic sectional curvature \(\kappa\) with respect to \(\mathcal D_s^t\), where \((t,s)\in\Omega\). Then either\/ {\rm (i)} \(h\) is K\"ahler, or\/
{\rm (ii)} a holomorphic cover of \((M,h)\) is an isosceles Hopf manifold with an admissible metric, and in this case $(t,s)\in \Gamma$ on the Chen-Nie curve. Moreover, if \(\kappa\) is globally constant \(c\), then in case\/ {\rm (ii)} the admissible metric is the standard Hopf metric up to scaling and \(c=0\).
\end{theorem}

Note that, for the globally constant part of Theorem~\ref{thm:two-parameter-classification},
the case of the Chern connection with \(c\leq0\) was previously proved
by Chen, Chen, and Nie in \cite{ChenChenNie}, while the remaining case \(c>0\) was recently settled by Huang and Wan in \cite{HuangWan}. Moreover, the cases of Levi--Civita and Bismut connections were obtained by Chen and Zheng \cite{ChenZhengLCK}.  Our proof of Theorem~\ref{thm:two-parameter-classification} uses the techniques in \cite{HuangWan,ChenZhengLCK}.  We also use the formulas in \cite{ChenNie}.

In Section~\ref{sec:preliminaries}, we introduce the notation and
recall the basic facts used in this paper.
Sections~\ref{sec:curvature}--\ref{sec:strict} contain the proof of
Theorem~\ref{thm:main-classification}.
Section~\ref{sec:two-parameter} contains the proof of
Theorem~\ref{thm:two-parameter-classification}.
Sections~\ref{sec:balanced-curvature} and
\ref{sec:balanced-torsion} contain the proof of
Theorem~\ref{thm:balanced-main}.

\section{Preliminaries}\label{sec:preliminaries}

In this section, we fix the notation and recall the basic facts used below.

\subsection{Curvature, torsion, and holomorphic sectional curvature}

Let \((M^n,J,h)\) be a Hermitian manifold.  The connections
\(\nabla^t\) and \(\mathcal D_s^t\) in
\eqref{eq:gauduchon-intro}--\eqref{eq:two-parameter-intro} are metric
connections on \(TM\).  The Chern torsion is
\[
 T(X,Y):=\nabla_X^cY-\nabla_Y^cX-[X,Y].
\]
Thus \(T\) always denotes the Chern torsion.  For a metric connection
\(\nabla\), we use the curvature convention in \cite{ChenNie}:
\begin{align*}
 R^\nabla(X,Y)Z
 &=\nabla_X\nabla_YZ-\nabla_Y\nabla_XZ-\nabla_{[X,Y]}Z,\\
 R^\nabla(X,Y,Z,W)
 &=h\bigl(R^\nabla(X,Y)Z,W\bigr).
\end{align*}
For a Hermitian metric \(h\), \(R^{LC}\) denotes its Levi--Civita
curvature.  We also write
\[
 R^c:=R^{\nabla^c},\qquad
 R^t:=R^{\nabla^t},\qquad
 R^D:=R^{\mathcal D_s^t}.
\]
Let \(\{e_i\}_{i=1}^n\) be a local \(h\)-unitary frame of
\(T^{1,0}M\).  Repeated indices are summed, and
\[
 R^\nabla_{i\bar j k\bar\ell}
 :=R^\nabla(e_i,\bar e_j,e_k,\bar e_\ell),
 \qquad
 T(e_i,e_k)=T^j_{ik}e_j,
 \qquad
 T(e_i,\bar e_j)=0,
\]
where \(T^j_{ik}=-T^j_{ki}\).

Let \(\Sigma\subset T_xM\) be the \(J\)-invariant plane spanned by
\(\{X,JX\}\), where \(X\neq0\), and put
\(v=\frac12(X-\ii JX)\).  The holomorphic sectional curvature of
\(\nabla\) is
\begin{equation}\label{eq:hsc-intro}
 H_h^\nabla(\Sigma)=H_h^\nabla(v)
 :=\frac{R^\nabla(X,JX,JX,X)}{|X|_h^4}
 =\frac{R^\nabla(v,\bar v,v,\bar v)}{|v|_h^4},
 \qquad |v|_h^2=h(v,\bar v).
\end{equation}
We say that \(H_h^\nabla\) is \emph{pointwise constant} if there is a
smooth function \(\kappa:M\to\R\) such that
\begin{equation}\label{eq:pointwise-intro}
 H_h^\nabla(v)=\kappa(x)
 \qquad\text{for every }0\neq v\in T_x^{1,0}M.
\end{equation}

We define the symmetrization \(\widehat R^\nabla\) of the curvature tensor
\(R^\nabla\) by
\begin{equation}\label{eq:symmetrization}
\widehat R^\nabla_{i\bar j k\bar\ell}
=
\frac14\left(
R^\nabla_{i\bar j k\bar\ell}
+R^\nabla_{k\bar j i\bar\ell}
+R^\nabla_{i\bar\ell k\bar j}
+R^\nabla_{k\bar\ell i\bar j}
\right).
\end{equation}
We use the superscripts \(c,t,D\) in the same way for \(H_h\) and
\(\widehat R\).

By polarization \cite{Balas,ChenNie}, pointwise constancy is equivalent
to
\begin{equation}\label{eq:pointwise-hsc-h}
 \widehat R^\nabla_{i\bar j k\bar\ell}
 =
 \frac{\kappa}{2}
 \left(
 \delta_{ij}\delta_{k\ell}+\delta_{i\ell}\delta_{kj}
 \right)
\end{equation}
in every \(h\)-unitary frame.  When \(\kappa\equiv c\in\R\), the
holomorphic sectional curvature is globally constant.

On a Kähler manifold with metric \(g\), let
\(\{e_i\}_{i=1}^n\) be a local \(g\)-unitary frame of \(T^{1,0}M\).
We write
\[
 f_i=e_i(f),\qquad
 f_{\bar i}=\bar e_i(f),\qquad
 f_{i\bar j}=(\nabla^gdf)(e_i,\bar e_j),
\]
and
\[
 |\partial f|_g^2=\sum_i f_i f_{\bar i},
 \qquad
 \Delta_gf=\sum_i f_{i\bar i},
\]
where
\[
 (\nabla^gdf)(X,Y)=X\bigl(Y(f)\bigr)-df(\nabla^g_XY).
\]

\subsection{Locally conformally K\"ahler metrics and the universal cover}
\label{subsec:cover}

\begin{definition}\label{def:lck}
Let \((M^n,J,h)\), \(n\geq2\), be a Hermitian manifold, and let
\(\omega_h(X,Y)=h(JX,Y)\) be its fundamental form.  The metric \(h\) is
called \emph{locally conformally K\"ahler} (LCK for short) if there is
a closed real one-form \(\theta\) such that
\begin{equation}\label{eq:lck-definition}
 d\omega_h=\theta\wedge\omega_h,
 \qquad d\theta=0.
\end{equation}
The one-form \(\theta\), which is unique for \(n\geq2\), is called the
\emph{Lee form}.  An LCK metric is called \emph{globally conformally
K\"ahler} (GCK for short) if \(\theta\) is exact.  It is called
\emph{strictly LCK} if \(\theta\) is not exact, or equivalently if
\([\theta]\neq0\) in \(H^1(M,\R)\).
\end{definition}

Indeed, the Poincar\'e lemma gives \(\theta=du\) locally, and
\begin{equation}\label{eq:local-kahler-rescaling}
 d(e^{-u}\omega_h)
 =e^{-u}\bigl(d\omega_h-du\wedge\omega_h\bigr)=0.
\end{equation}
Thus \(e^{-u}h\) is K\"ahler locally.  The local functions \(u\) can be
chosen as restrictions of one \(u\in C^\infty(M,\R)\) if and only if
\(\theta\) is exact.  In this case
\eqref{eq:local-kahler-rescaling} holds on \(M\), so \(e^{-u}h\) is
K\"ahler and \(h\) is GCK.  Finally, \(h\) itself is K\"ahler if and
only if \(\theta=0\); an exact nonzero Lee form gives a non-K\"ahler GCK
metric.

Assume from now on that \((M,J,h)\) is a connected LCK manifold, and let
\[
    \pi:\widetilde M\longrightarrow M
\]
be the universal covering.  Set
\[
    \widetilde\theta:=\pi^*\theta,
    \qquad
    \widetilde\omega:=\pi^*\omega_h.
\]
Since pullback commutes with the exterior derivative, the LCK identities
\eqref{eq:lck-definition} give
\[
    d\widetilde\theta=0,
    \qquad
    d\widetilde\omega
    =\widetilde\theta\wedge\widetilde\omega.
\]
Fixing $\widetilde x_0\in\widetilde M$, define
\[
f(\widetilde x)
:=
\frac12\int_{\widetilde x_0}^{\widetilde x}\widetilde\theta.
\]
Since $\widetilde\theta$ is closed and $\widetilde M$ is simply
connected, $\widetilde\theta$ is exact. Hence the integral is
independent of the chosen path, and $f$ is a globally defined smooth
function satisfying
\begin{equation}\label{eq:Lee-cover-function}
    \pi^*\theta=\widetilde\theta=2\,df.
\end{equation}
Now set
\[
    g:=e^{-2f}\pi^*h,
    \qquad
    \omega_g:=e^{-2f}\widetilde\omega.
\]
Then
\[
 d\omega_g
 =e^{-2f}\bigl(d\widetilde\omega-2\,df\wedge\widetilde\omega\bigr)
 =e^{-2f}\bigl(\widetilde\theta-2\,df\bigr)\wedge\widetilde\omega
 =0.
\]
Hence \(g\) is K\"ahler.  Consequently, the pullback metric
\(\widetilde h:=\pi^*h\) is globally conformally K\"ahler on
\(\widetilde M\), since
\begin{equation}\label{eq:universal-conformal}
 \widetilde h=\pi^*h=e^{2f}g.
\end{equation}

Let \(\gamma\) be a deck transformation.  Since
\(\pi\circ\gamma=\pi\), equation
\eqref{eq:Lee-cover-function} gives
\[
 d(f\circ\gamma-f)
 =\gamma^*df-df
 =\frac12\left((\pi\circ\gamma)^*\theta-\pi^*\theta\right)
 =0.
\]
Thus \(f\circ\gamma-f\) is constant.  Define the positive constant
\[
 r_\gamma:=e^{\,f-f\circ\gamma}.
\]
Applying \(\gamma^*\) to \eqref{eq:universal-conformal}, we obtain
\[
 e^{2f}g
 =\pi^*h
 =(\pi\circ\gamma)^*h
 =\gamma^*(\pi^*h)
 =\gamma^*(e^{2f}g)
 =e^{2f\circ\gamma}\gamma^*g.
\]
Therefore
\begin{equation}\label{eq:deck-homothety}
 \gamma^*g
 =e^{2(f-f\circ\gamma)}g
 =r_\gamma^2g,
 \qquad r_\gamma>0.
\end{equation}
By the definition of \(r_\gamma\),
\begin{equation}\label{eq:f-automorphy}
 f\circ\gamma=f-\log r_\gamma.
\end{equation}
For any two deck transformations \(\gamma_1\) and \(\gamma_2\),
\[
 f\circ\gamma_1\circ\gamma_2
 =f-\log r_{\gamma_1}-\log r_{\gamma_2}.
\]
Hence
\[
 r_{\gamma_1\circ\gamma_2}
 =r_{\gamma_1}r_{\gamma_2},
\]
so \(\gamma\mapsto r_\gamma\) is a multiplicative character.

If \(\theta=du\) on \(M\), then
\eqref{eq:Lee-cover-function} gives
\(d(2f-\pi^*u)=0\), so \(f\) is deck invariant and every
\(r_\gamma=1\).  Conversely, if every \(r_\gamma=1\), then \(f\)
descends to a function \(\bar f\) on \(M\), and
\eqref{eq:Lee-cover-function} gives \(\theta=2\,d\bar f\).  Hence the Lee
form is exact if and only if the homothety character is trivial.
Thus, when the Lee form is not exact, some deck transformation has
\(r_\gamma\neq1\).  Replacing it by its inverse when necessary, we may
choose one satisfying
\begin{equation}\label{eq:strict-contraction}
    0<r_\gamma<1.
\end{equation}

If \(H_h^t\) is pointwise constant with curvature function \(\kappa\),
set \(\widetilde\kappa:=\pi^*\kappa\).  The pullback of \(\nabla^t\) is
the \(t\)-Gauduchon connection of \(\widetilde h\), and its holomorphic
sectional curvature is pointwise constant with function
\(\widetilde\kappa\).  The analogous statement holds for
\(\mathcal D_s^t\), since pullback commutes with the affine construction
\eqref{eq:two-parameter-intro}.

\subsection{Isosceles Hopf manifolds}
Let $D=\operatorname{diag}(d_1,\dots,d_n)$, where
$|d_1|=\cdots=|d_n|=r<1$, and set
$M_D=(\C^n\setminus\{0\})/\langle z\mapsto Dz\rangle$.
Such a manifold is called an \emph{isosceles Hopf manifold}. Let
$q_D:\C^n\setminus\{0\}\longrightarrow M_D$ be the quotient map, and let
$g_0$ be the standard Euclidean Hermitian metric on $\C^n$.

\begin{definition}\label{def:admissible}
A Hermitian metric \(h\) on \(M_D\) is called \emph{admissible} if
there exist a constant \(c_0>0\) and a complex matrix \(\mathsf A\)
such that
\begin{equation}\label{eq:admissible-statement}
 q_D^*h
 =\frac{c_0}{Q_{\mathsf A}(z)}g_0,
 \qquad
 Q_{\mathsf A}(z)
 =|z|^2+z^{\mathsf T}\mathsf A z
  +\overline{z^{\mathsf T}\mathsf A z},
\end{equation}
where
\begin{equation}\label{eq:admissible-conditions-statement}
 \mathsf A^{\mathsf T}=\mathsf A,
 \qquad
 \mathsf A\overline{\mathsf A}<\frac14I,
 \qquad
 D^{\mathsf T}\mathsf A D=r^2\mathsf A.
\end{equation}
\end{definition}

The middle inequality means that
\(\frac14I-\mathsf A\overline{\mathsf A}\) is positive definite, and
it guarantees that \(Q_{\mathsf A}>0\).  The last condition gives
\[
 Q_{\mathsf A}(Dz)=r^2Q_{\mathsf A}(z).
\]
Since \(D^*g_0=r^2g_0\), we have
\[
 D^*\left(\frac{c_0}{Q_{\mathsf A}}g_0\right)
 =
 \frac{c_0}{Q_{\mathsf A}\circ D}D^*g_0
 =
 \frac{c_0}{r^2Q_{\mathsf A}}\,r^2g_0
 =
 \frac{c_0}{Q_{\mathsf A}}g_0.
\]
Thus \(\frac{c_0}{Q_{\mathsf A}}g_0\) is invariant under
\(\langle D\rangle\) and defines a Hermitian metric on \(M_D\).
See \cite{ChenNie}.

The standard Hopf metric \(h_{\mathrm H}\) on \(M_D\) is defined by
\[
 q_D^*h_{\mathrm H}=\frac{g_0}{|z|^2}.
\]
Since \(|z|^2/Q_{\mathsf A}(z)\) is \(D\)-invariant, every admissible
metric is globally conformal to \(h_{\mathrm H}\).

\section{The conformal curvature identity and Bochner-flatness}
\label{sec:curvature}

\subsection{The conformal curvature formula}

We now specialize to the Gauduchon family.  Assume that \(H_h^t\) is
pointwise constant with curvature function \(\kappa\), and retain the
notation of the preceding section.

Let \(\{e_1,\dots,e_n\}\) be a local \(g\)-unitary frame and put
\(\widetilde e_i=e^{-f}e_i\), so that \(\{\widetilde e_i\}\) is
\(\widetilde h\)-unitary.  Let \(R^g\) denote the K\"ahler curvature of
\(g\), and let \(\widetilde R^t\) denote the curvature of the
\(t\)-Gauduchon connection of \(\widetilde h\).  Since \(g\) is
K\"ahler, formula (4.10) in \cite{ChenNie} becomes
\begin{equation}\label{eq:CN410}
\begin{aligned}
e^{2f}\widehat{\widetilde R}^{\,t}_{k\bar l i\bar j}
-R^g_{k\bar l i\bar j}
={}&
-\frac12\left(
 f_{i\bar l}\delta_{jk}
+f_{k\bar l}\delta_{ij}
+f_{i\bar j}\delta_{kl}
+f_{k\bar j}\delta_{il}
\right)
\\
&+\frac{(1-t)^2}{4}\left(
 f_i f_{\bar j}\delta_{kl}
+f_k f_{\bar j}\delta_{il}
+f_i f_{\bar l}\delta_{kj}
+f_k f_{\bar l}\delta_{ij}
\right)
\\
&-\frac{(1-t)^2}{2}|\partial f|_g^2
\left(
 \delta_{kj}\delta_{il}
+\delta_{ij}\delta_{kl}
\right).
\end{aligned}
\end{equation}
Here all derivatives of \(f\) are taken with respect to the K\"ahler
connection of \(g\).

For any Hermitian metric \(k\), define
\begin{equation}\label{eq:G-def}
    (G_k)_{p\bar q i\bar j}
    =k_{i\bar j}k_{p\bar q}+k_{i\bar q}k_{p\bar j},
\end{equation}
and, for a Hermitian \((1,1)\)-tensor \(S\),
\begin{equation}\label{eq:L-def}
 (L_k(S))_{p\bar q i\bar j}
 =S_{i\bar q}k_{p\bar j}+S_{p\bar q}k_{i\bar j}
  +S_{i\bar j}k_{p\bar q}+S_{p\bar j}k_{i\bar q}.
\end{equation}
Set
\begin{equation}\label{eq:a-def}
    a:=\frac{(1-t)^2}{2}\geq0,
\end{equation}
and set
\begin{equation}\label{eq:tensor-A-def}
    A_{i\bar j}:=f_{i\bar j}-a f_i f_{\bar j}.
\end{equation}
Finally, define
\begin{equation}\label{eq:beta-t-def}
    \beta_t
    :=
    e^{2f}\widetilde\kappa
    +(1-t)^2|\partial f|_g^2
    =e^{2f}\widetilde\kappa+2a|\partial f|_g^2.
\end{equation}
By pointwise constancy,
\[
    \widehat{\widetilde R}^{\,t}_{k\bar l i\bar j}
    =
    \frac{\widetilde\kappa}{2}
    \left(\delta_{ij}\delta_{kl}+\delta_{il}\delta_{kj}\right).
\]
Substitution into \eqref{eq:CN410} gives
\begin{equation}\label{eq:curvature-identity}
    R^g=\frac{\beta_t}{2}G_g+\frac12L_g(A).
\end{equation}
Indeed, the terms on the right-hand side of \eqref{eq:CN410} can be rewritten as
\begin{align*}
&-\frac12L_g\bigl((f_{i\bar j})\bigr)
 +\frac{(1-t)^2}{4}L_g\bigl((f_i f_{\bar j})\bigr)
 -\frac{(1-t)^2}{2}|\partial f|_g^2G_g
=-\frac12L_g(A)-a|\partial f|_g^2G_g.
\end{align*}

\subsection{Injectivity of \texorpdfstring{\(L_g\)}{Lg}}

At a fixed point \(x\in\widetilde M\), let \(V=T_x\widetilde M\), and set
\[
 \mathcal H
 :=
 \left\{S\in\operatorname{Sym}^2V^*:
 S(JX,JY)=S(X,Y)\right\}.
\]
Thus \(\mathcal H\) is the real vector space of Hermitian symmetric
two-tensors.  Let \(\mathcal K\subset\otimes^4V^*\) be the real vector
space of tensors \(R\) satisfying
\[
\begin{aligned}
 \text{\rm(i)}\quad&
 R(X,Y,Z,W)=-R(Y,X,Z,W)=-R(X,Y,W,Z),\\
 \text{\rm(ii)}\quad&
 R(X,Y,Z,W)=R(Z,W,X,Y),\\
 \text{\rm(iii)}\quad&
 R(X,Y,Z,W)+R(Y,Z,X,W)+R(Z,X,Y,W)=0,\\
 \text{\rm(iv)}\quad&
 R(JX,JY,Z,W)=R(X,Y,Z,W)=R(X,Y,JZ,JW)
\end{aligned}
\]
for all \(X,Y,Z,W\in V\).  These are the usual algebraic curvature
symmetries, the first Bianchi identity, and the K\"ahler
\(J\)-invariance.  We call \(\mathcal K\) the space of algebraic
K\"ahler curvature tensors.  Since \(g\) is K\"ahler,
\(R^g(x)\in\mathcal K\).  Formula \eqref{eq:L-def} also gives
\(L_g(S)\in\mathcal K\) for every \(S\in\mathcal H\).

The following injectivity lemma is due to Huang and Wan
\cite{HuangWan}.

\begin{lemma}\label{lem:L-injective}
For \(n\geq2\), the linear map
\[
 L_g:\mathcal H\longrightarrow\mathcal K
\]
is injective.
\end{lemma}

\begin{proof}
Suppose \(L_g(S)=0\).  Contracting the \(k,\bar l\) pair gives
\begin{equation}\label{eq:L-contract}
    (n+2)S+(\operatorname{tr}_gS)g=0.
\end{equation}
Indeed, in a unitary frame,
\begin{align*}
\sum_{k=1}^n(L_g(S))_{k\bar k i\bar j}
&=
\sum_k\left(
S_{i\bar k}\delta_{kj}
+S_{k\bar k}\delta_{ij}
+S_{i\bar j}\delta_{kk}
+S_{k\bar j}\delta_{ik}
\right)=(n+2)S_{i\bar j}
+(\operatorname{tr}_gS)\delta_{ij}.
\end{align*}
Taking one more trace yields
\[
    2(n+1)\operatorname{tr}_gS=0.
\]
Thus \(\operatorname{tr}_gS=0\), and \eqref{eq:L-contract} gives
\(S=0\).
\end{proof}

\subsection{The K\"ahler curvature decomposition and Bochner-flatness}
\label{subsec:bochner}

For any Hermitian \((1,1)\)-tensor \(S\), define its \(g\)-trace-free
part by
\[
    S^\circ
    :=S-\frac{\trg S}{n}g.
\]
In particular,
\[
    (\Ric^g)^\circ
    =\Ric^g-\frac{\trg(\Ric^g)}{n}g.
\]
Applying this notation to \(A\), write
\begin{equation}\label{eq:A-splitting}
    A=A^\circ+\frac{\trg A}{n}g.
\end{equation}
Since
\begin{equation}\label{eq:Lg-equals-2G}
    L_g(g)=2G_g,
\end{equation}
identity \eqref{eq:curvature-identity} becomes
\begin{equation}\label{eq:curvature-splitting}
    R^g
    =
    \frac12L_g(A^\circ)
    +
    \left(\frac{\beta_t}{2}
    +\frac{\trg A}{n}\right)G_g.
\end{equation}

In a \(g\)-unitary basis, define the Ricci and scalar contractions of
\(R\in\mathcal K\) by
\[
 \Ric(R)_{i\bar j}
 :=
 \sum_{k=1}^n R_{k\bar k i\bar j},
 \qquad
 s(R)
 :=
 \sum_{i=1}^n\Ric(R)_{i\bar i}.
\]
For \(R^g\), we write
\[
 \Ric^g:=\Ric(R^g),
 \qquad
 s_g:=s(R^g).
\]
Set
\[
 \mathcal H_0
 :=
 \left\{
 S\in\mathcal H:
 \sum_{i=1}^n S_{i\bar i}=0
 \right\},
\]
and
\[
 \mathcal B
 :=
 \left\{
 R\in\mathcal K:
 \sum_{k=1}^nR_{k\bar k i\bar j}=0
 \text{ for all }1\leq i,j\leq n
 \right\}.
\]
The metric \(g\) induces an inner product on \(\mathcal K\).  The
Bochner tensor was introduced in \cite{Bochner}.  The corresponding
\(U(n)\)-irreducible orthogonal decomposition was obtained in
\cite{Alekseevskii}; see also \cite{HuangWan}.  In the notation above,
\begin{equation}\label{eq:kahler-decomposition}
 \mathcal K
 =
 \mathcal B
 \oplus L_g(\mathcal H_0)
 \oplus\R G_g.
\end{equation}
Here
\(L_g(\mathcal H_0)=\{L_g(S):S\in\mathcal H_0\}\).
Let \(P_{\mathcal B}:\mathcal K\to\mathcal B\) be the orthogonal
projection.  By definition, the Bochner curvature tensor of \(g\) is
\[
 \Boch(g):=P_{\mathcal B}(R^g).
\]
A K\"ahler metric is called \emph{Bochner--K\"ahler}, or
\emph{Bochner-flat}, if \(\Boch(g)=0\).  Since
\(A^\circ\in\mathcal H_0\), the right-hand side of
\eqref{eq:curvature-splitting} lies in
\(L_g(\mathcal H_0)\oplus\R G_g\).  Therefore its orthogonal projection
onto \(\mathcal B\) vanishes:
\begin{equation}\label{eq:bochner-zero}
    P_{\mathcal B}(R^g)=\Boch(g)=0.
\end{equation}
Thus the universal K\"ahler metric \(g\) is Bochner--K\"ahler.

Contracting \eqref{eq:curvature-identity} gives
\begin{equation}\label{eq:ricci-general}
    \Ric^g
    =
    \frac{n+1}{2}\beta_t g
    +\frac12\bigl((n+2)A+(\trg A)g\bigr),
\end{equation}
and therefore
\begin{equation}\label{eq:ricci0-general}
    (\Ric^g)^\circ=\frac{n+2}{2}A^\circ.
\end{equation}
The scalar contraction is
\begin{equation}\label{eq:scalar-general}
    s_g
    =
    (n+1)\left(\frac n2\beta_t+\trg A\right).
\end{equation}

We recall from \cite{HuangWan} the following facts about the
Tricerri--Vanhecke Bochner tensor.  For a Hermitian metric \(k\), let
\(B^{\mathrm{TV}}(k)\) denote the Bochner component in the
Tricerri--Vanhecke \(U(n)\)-decomposition of the Riemannian curvature
tensor of \(k\).  The metric \(k\) is called
\emph{Bochner-flat in the sense of Tricerri--Vanhecke} if
\[
 B^{\mathrm{TV}}(k)=0.
\]
This condition is preserved under Hermitian conformal changes.  When
\(k\) is K\"ahler, it is equivalent to the vanishing of the ordinary
K\"ahler Bochner tensor.  In complex dimension two, we use the
corrected four-dimensional formula for \(B^{\mathrm{TV}}\).

We now pass from \(g\) to \(h\).  Naturality under pullback by local
diffeomorphisms gives
\[
 B^{\mathrm{TV}}(\widetilde h)
 =
 B^{\mathrm{TV}}(\pi^*h)
 =
 \pi^*B^{\mathrm{TV}}(h).
\]
Since \(\pi\) is surjective, the first equivalence below follows.  The
second follows from \(\widetilde h=e^{2f}g\) and the conformal
invariance above.  The last follows because \(g\) is K\"ahler.  Hence
\begin{equation}\label{eq:TV-equivalence}
 B^{\mathrm{TV}}(h)=0
 \quad\Longleftrightarrow\quad
 B^{\mathrm{TV}}(\widetilde h)=0
 \quad\Longleftrightarrow\quad
 B^{\mathrm{TV}}(g)=0
 \quad\Longleftrightarrow\quad
 \Boch(g)=0.
\end{equation}
In particular, \eqref{eq:bochner-zero} gives
\(B^{\mathrm{TV}}(h)=0\).

\section{The globally conformally K\"ahler case}\label{sec:gck}

Assume that the Lee form \(\theta\) is exact.  By
Definition~\ref{def:lck}, this is precisely the GCK case.  Choose
\(f\in C^\infty(M,\R)\) such that \(\theta=2\,df\).  Then
\eqref{eq:local-kahler-rescaling} shows that
\(g_M:=e^{-2f}h\) is a globally defined K\"ahler metric on \(M\), and
\begin{equation}\label{eq:GCK-conformal}
    h=e^{2f}g_M.
\end{equation}
The computation in Section~\ref{sec:curvature}, with \(g_M\) in place
of \(g\), applies directly on \(M\) and gives \(\Boch(g_M)=0\).

\subsection{Compact Bochner--K\"ahler rigidity}

Thus $(M,g_M)$ is a connected compact Bochner--K\"ahler manifold.
By Bryant's compact classification \cite{Bryant}, it is a compact
quotient of a Hermitian symmetric product
\[
M_c^p\times M_{-c}^{\,n-p},
\]
where $M_c^q$ denotes the simply connected complex
$q$-dimensional K\"ahler space form of constant holomorphic sectional
curvature $c$. In particular, $g_M$ is locally symmetric. Consequently,
\begin{equation}\label{eq:parallel-curvature}
    \nabla^{g_M}R^{g_M}=0.
\end{equation}

Since contraction commutes with \(\nabla^{g_M}\) and
\(\nabla^{g_M}g_M=0\),
\[
    \nabla^{g_M}\Ric^{g_M}=0,
    \qquad
    \nabla^{g_M}s_{g_M}=0,
    \qquad
    \nabla^{g_M}(\Ric^{g_M})^\circ=0.
\]
Combining this with \eqref{eq:ricci0-general}, applied to \(g_M\),
gives
\begin{equation}\label{eq:parallel-A0}
    \nabla^{g_M}A^\circ
    =
    \frac{2}{n+2}\nabla^{g_M}(\Ric^{g_M})^\circ
    =0.
\end{equation}

\subsection{Vanishing of the trace-free part when
\texorpdfstring{\(t=1\)}{t=1}}

When \(t=1\), one has \(a=0\).  Thus
\eqref{eq:tensor-A-def} gives
\[
 A_{i\bar j}=f_{i\bar j}.
\]
Put \(P:=A^\circ\).  Since
\[
 A
 =
 A^\circ+\frac{\operatorname{tr}_{g_M}A}{n}g_M
 \qquad\text{and}\qquad
 \operatorname{tr}_{g_M}A=\Delta_{g_M}f,
\]
we obtain
\[
 f_{i\bar j}
 =
 P_{i\bar j}
 +\frac{\Delta_{g_M}f}{n}(g_M)_{i\bar j}.
\]
By \eqref{eq:parallel-A0}, \(P\) is parallel. Applying the divergence theorem to \(P^{i\bar j}f_{\bar j}e_i\) , we obtain
\begin{align}
0
&=\int_M\nabla^{g_M}_i
  \bigl(P^{i\bar j}f_{\bar j}\bigr)\,dV_{g_M}
 \nonumber\\
&=\int_M P^{i\bar j}f_{i\bar j}\,dV_{g_M}
 \nonumber\\
&=\int_M P^{i\bar j}
  \left(
  P_{i\bar j}
  +\frac{\Delta_{g_M}f}{n}(g_M)_{i\bar j}
  \right)dV_{g_M}
 \nonumber\\
&=\int_M|P|_{g_M}^2\,dV_{g_M}.
\label{eq:GCK-t1-integral}
\end{align}
Here we used \(\nabla^{g_M}P=0\) and
\(\operatorname{tr}_{g_M}P=0\).  Since the integrand is nonnegative,
\(P=0\), and hence
\begin{equation}\label{eq:GCK-A0-zero-t1}
 A^\circ=0.
\end{equation}

\subsection{Vanishing of the trace-free part when \texorpdfstring{\(t\neq1\)}{t not equal to 1}}

Now assume \(t\neq1\), so \(a>0\), and define
\begin{equation}\label{eq:GCK-xi}
    \xi:=e^{-af}>0.
\end{equation}
Then
\begin{equation}\label{eq:GCK-xi-hessian}
 \xi_{i\bar j}
 =-a\xi f_{i\bar j}+a^2\xi f_if_{\bar j}
 =-a\xi A_{i\bar j}.
\end{equation}
Again put \(P:=A^\circ\).  Applying the divergence theorem to
\(P^{i\bar j}\xi_{\bar j}e_i\), and using that \(P\) is parallel, gives
\begin{align}
0
&=\int_M\nabla^{g_M}_i
  \bigl(P^{i\bar j}\xi_{\bar j}\bigr)\,dV_{g_M}
 \nonumber\\
&=\int_M P^{i\bar j}\xi_{i\bar j}\,dV_{g_M}
 \nonumber\\
&=-a\int_M\xi\,P^{i\bar j}A_{i\bar j}\,dV_{g_M}
 \nonumber\\
&=-a\int_M\xi|P|_{g_M}^2\,dV_{g_M}.
\label{eq:GCK-tnot1-integral}
\end{align}
Here the third equality follows from
\eqref{eq:GCK-xi-hessian}, and the last uses
\(\operatorname{tr}_{g_M}P=0\).  Since \(a>0\) and \(\xi>0\),
\begin{equation}\label{eq:GCK-A0-zero-tnot1}
 A^\circ=0.
\end{equation}

Thus, for every \(t\),
\begin{equation}\label{eq:GCK-A-pure-trace}
    A=\lambda g_M,
\end{equation}
for some real-valued function \(\lambda\).

\subsection{Constancy of the conformal factor}
\label{subsec:gck-conformal-constancy}

Define
\begin{equation}\label{eq:GCK-u}
    u=
    \begin{cases}
    f,&t=1,\\[1mm]
    e^{-af},&t\neq1.
    \end{cases}
\end{equation}
Equations \eqref{eq:GCK-A-pure-trace} and
\eqref{eq:GCK-xi-hessian} show that there is a real function \(F\) such
that
\begin{equation}\label{eq:GCK-u-Hessian}
    u_{i\bar j}=F(g_M)_{i\bar j},
\end{equation}
or equivalently
\begin{equation}\label{eq:GCK-ddbar}
    \ddbar u=F\omega_{g_M}.
\end{equation}
More explicitly, $F=\lambda$ when $t=1$, whereas
$F=-ae^{-af}\lambda$ when $t\neq 1$. Taking the exterior derivative
and using $d\omega_{g_M}=0$, we obtain
\[
dF\wedge\omega_{g_M}=0.
\]
Wedging with $\omega_{g_M}^{n-2}$ and applying the pointwise linear
Lefschetz isomorphism
\[
L^{n-1}\colon
\Lambda^1T_x^*M
\longrightarrow
\Lambda^{2n-1}T_x^*M,
\qquad
\eta\longmapsto\eta\wedge\omega_{g_M}^{n-1},
\]
we conclude that $dF=0$.
Taking the trace of \eqref{eq:GCK-u-Hessian} and integrating over \(M\) gives
\[
    0=\int_M\Delta_{g_M}u\,dV_{g_M}
    =nF\operatorname{Vol}_{g_M}(M),
\]
so \(F=0\).  Hence \(u\) is harmonic and is constant by the strong maximum
principle on the connected compact manifold.  It follows that \(f\) is
constant in both cases (for \(t\neq1\), use
\(f=-a^{-1}\log u\)).  Therefore \(h\) is K\"ahler.

For a K\"ahler metric all Gauduchon connections coincide with the
K\"ahler connection, and the curvature already has the K\"ahler
symmetries.  Consequently \(\widehat R^t=R^h\), and
\eqref{eq:pointwise-hsc-h} directly gives
\begin{equation}\label{eq:GCK-space-form-tensor}
    R^h=\frac{\kappa}{2}G_h.
\end{equation}
Contracting gives
\[
    \Ric^h=\frac{n+1}{2}\kappa h.
\]
The K\"ahler Ricci form is
\[
 \rho_h
 :=\sqrt{-1}\,\Ric^h_{i\bar j}\,dz^i\wedge d\bar z^j
 =\frac{n+1}{2}\kappa\,\omega_h.
\]
By the local Chern--Ricci formula
\(\rho_h=-\sqrt{-1}\,\partial\bar\partial
\log\det(h_{i\bar j})\), the form \(\rho_h\) is closed; hence
\[
    0=d\rho_h=\frac{n+1}{2}\,d\kappa\wedge\omega_h.
\]
Wedging with \(\omega_h^{n-2}\) and using the same pointwise Lefschetz
isomorphism gives \(d\kappa=0\).  Thus \(\kappa\) is constant, and
\eqref{eq:GCK-space-form-tensor} shows that \(h\) is a complex space
form.

\section{The strict LCK and globally constant cases}\label{sec:strict}

\subsection{Euclidean uniformization}\label{subsec:uniformization}

Assume from now on that the Lee form is not exact.  By
\eqref{eq:TV-equivalence}, the Tricerri--Vanhecke Bochner tensor
vanishes.  The following uniformization result follows from Fried's
classification of closed similarity manifolds and Kamishima's work on
Bochner-flat LCK manifolds \cite{Fried,Kamishima}; see also
\cite{HuangWan}.

\begin{proposition}\label{prop:uniformization}
Let \((M^n,h)\) be a connected compact strict LCK manifold, \(n\geq2\),
whose Tricerri--Vanhecke Bochner tensor vanishes.  After multiplying the
universal K\"ahler metric by a positive constant, there is a holomorphic
isometry
\begin{equation}\label{eq:euclidean-uniformization}
    (\widetilde M,g)
    \cong
    (\C^n\setminus\{0\},g_0).
\end{equation}
Under this identification, every deck transformation is of the form
\begin{equation}\label{eq:similarity-deck}
    \gamma(z)=r_\gamma U_\gamma z,
    \qquad
    r_\gamma>0,
    \quad U_\gamma\in U(n),
\end{equation}
and at least one deck transformation satisfies \(r_\gamma\neq1\).
\end{proposition}

We absorb the constant rescaling of \(g\) into \(f\): if
\(g\) is replaced by \(cg\), replace \(f\) by
\(f-\frac12\log c\).  This leaves \(\pi^*h=e^{2f}g\) unchanged.
We retain the notation \(f\) after identifying the normalized metric
with \(g_0\).

Choose a deck transformation \(\gamma\) with
\(0<r:=r_\gamma<1\), and write
\begin{equation}\label{eq:chosen-gamma}
    \gamma(z)=rUz,
    \qquad U\in U(n).
\end{equation}
The automorphy relation \eqref{eq:f-automorphy} becomes
\begin{equation}\label{eq:flat-f-automorphy}
    f(rUz)=f(z)-\log r.
\end{equation}

\subsection{The flat equation and the exceptional parameters}
\label{subsec:flat}

On the Euclidean cover, \(R^{g_0}=0\), so
\eqref{eq:curvature-identity} reduces to
\begin{equation}\label{eq:flat-curvature-identity}
    0=\frac{\beta_t}{2}G_{g_0}+\frac12L_{g_0}(A).
\end{equation}
Since \(L_{g_0}(g_0)=2G_{g_0}\), this can be rewritten as
\[
 0
 =
 \frac12L_{g_0}\left(A+\frac{\beta_t}{2}g_0\right).
\]
Lemma~\ref{lem:L-injective} therefore gives
\begin{equation}\label{eq:A-flat}
    A=-\frac{\beta_t}{2}g_0.
\end{equation}
Equivalently,
\begin{equation}\label{eq:flat-master-PDE}
    f_{i\bar j}-af_if_{\bar j}
    =
    -\frac12\left(
    e^{2f}\widetilde\kappa+2a|\partial f|_{g_0}^2
    \right)\delta_{ij}.
\end{equation}
We use the following lemma.

\begin{lemma}\label{lem:constant-lambda}
Let \((N^n,g)\), \(n\geq2\), be a connected K\"ahler manifold, and let
\(u\in C^\infty(N,\R)\).  In a local unitary frame
\(\{e_1,\ldots,e_n\}\), write
\(u_{i\bar j}=(\nabla^2u)(e_i,\bar e_j)\), where \(\nabla\) is the
K\"ahler connection, and use semicolons for further covariant
derivatives.  If
\begin{equation}\label{eq:lambda-Hessian}
    u_{i\bar j}=\lambda\delta_{ij}
\end{equation}
for a real-valued function \(\lambda\), then \(\lambda\) is constant.
\end{lemma}

\begin{proof}
By the Ricci identities for the K\"ahler connection and their complex
conjugates,
\begin{equation}\label{eq:ricci-commutation-u}
    u_{i\bar j;m}=u_{m\bar j;i},
    \qquad
    u_{i\bar j;\bar m}=u_{i\bar m;\bar j}.
\end{equation}
At any point, choose a normal unitary frame.  For each \(m\), choose
\(i\neq m\).  Then
\[
    \lambda_m
    =u_{i\bar i;m}
    =u_{m\bar i;i}
    =(\lambda\delta_{mi})_{;i}
    =0.
\]
Taking complex conjugates gives \(\lambda_{\bar m}=0\).  Hence
\(d\lambda=0\).
\end{proof}

\begin{lemma}\label{lem:positive-pluriharmonic}
Let \(v\) be a real pluriharmonic function on
\(\C^n\setminus\{0\}\), \(n\geq2\).  If \(v>0\), then \(v\) is
constant.  Moreover, if
\[
    v(rUz)=v(z)+c
\]
for some \(0<r<1\), \(U\in U(n)\), and \(c\in\R\), then \(c=0\).
\end{lemma}

\begin{proof}
The form \(2\partial v\) is closed and holomorphic.  Since
\(\C^n\setminus\{0\}\) is simply connected, the
Poincar\'e lemma gives a holomorphic function \(\varphi\) with
\(d\varphi=2\partial v\) and \(v=\operatorname{Re}\varphi\), after
adding a real constant.  Hartogs' extension theorem extends
\(\varphi\) to \(\C^n\).  If \(v>0\), then \(e^{-\varphi}\) is entire
and
\[
    \bigl|e^{-\varphi}\bigr|
    =e^{-\operatorname{Re}\varphi}
    =e^{-v}\leq1.
\]
Liouville's theorem shows that it is constant, and hence \(v\) is
constant.  For the second assertion,
\[
    \operatorname{Re}\bigl(\varphi(rUz)-\varphi(z)\bigr)=c.
\]
Letting \(z\to0\) gives \(c=0\).
\end{proof}

\subsubsection{Exclusion of the Chern parameter
\texorpdfstring{\(t=1\)}{t=1}}
\label{subsec:exclude-chern}

Suppose \(t=1\).  Then \(a=0\), and \eqref{eq:flat-master-PDE} has the
form
\begin{equation}\label{eq:flat-t1-Hessian}
    f_{i\bar j}=\lambda\delta_{ij},
    \qquad
    \lambda:=-\frac12e^{2f}\widetilde\kappa.
\end{equation}
By Lemma~\ref{lem:constant-lambda}, \(\lambda\) is constant.
Taking the complex Hessian of the two sides of
\eqref{eq:flat-f-automorphy} gives
\[
 \frac{\partial^2(f\circ\gamma)}
 {\partial z_i\,\partial\bar z_j}
 =r^2\sum_{p,q}U_{pi}\overline{U_{qj}}\,
   f_{p\bar q}(rUz)
 =r^2\lambda\delta_{ij},
 \qquad
 \frac{\partial^2(f-\log r)}
 {\partial z_i\,\partial\bar z_j}
 =\lambda\delta_{ij}.
\]
Therefore
\[
    (r^2-1)\lambda=0.
\]
Since \(r\neq1\), one has \(\lambda=0\).  Thus \(f\) is real
pluriharmonic on \(\C^n\setminus\{0\}\).  Equation
\eqref{eq:flat-f-automorphy} and the second assertion of
Lemma~\ref{lem:positive-pluriharmonic}, applied with \(v=f\) and
\(c=-\log r\), give \(-\log r=0\), contrary to \(0<r<1\).  Hence
\(t\neq1\).

\subsubsection{The case
\texorpdfstring{\(t\neq1\)}{t not equal to 1}}

Now suppose \(t\neq1\), so \(a>0\), and define
\begin{equation}\label{eq:strict-xi}
    \xi:=e^{-af}>0.
\end{equation}
Differentiating \(\xi=e^{-af}\), and then using the definition of \(A\)
and \eqref{eq:A-flat}, we obtain
\begin{equation}\label{eq:strict-xi-Hessian}
 \xi_{i\bar j}
 =-a\xi f_{i\bar j}+a^2\xi f_if_{\bar j}
 =-a\xi A_{i\bar j}
 =\frac{a\beta_t}{2}\xi\,\delta_{ij}.
\end{equation}
Put
\begin{equation}\label{eq:strict-lambda-def}
    \lambda:=\frac{a\beta_t}{2}\xi.
\end{equation}
Then
\begin{equation}\label{eq:strict-xi-Hessian-lambda}
    \xi_{i\bar j}=\lambda\delta_{ij}.
\end{equation}
Since \(\beta_t\) and \(\xi\) are real, so is \(\lambda\).
Lemma~\ref{lem:constant-lambda} shows that \(\lambda\) is
constant.  For later use, expand \eqref{eq:strict-lambda-def} using
\[
    e^{2f}=\xi^{-2/a},
    \qquad
    \partial\xi=-a\xi\,\partial f.
\]
This gives
\begin{equation}\label{eq:strict-lambda-expanded}
 \lambda
 =\frac a2\widetilde\kappa\,\xi e^{2f}
  +a^2\xi|\partial f|_{g_0}^2
 =\frac a2\widetilde\kappa\,\xi^{1-2/a}
  +\frac{|\partial\xi|_{g_0}^2}{\xi}.
\end{equation}

Equation \eqref{eq:flat-f-automorphy} gives explicitly
\begin{equation}\label{eq:strict-xi-automorphy}
    \xi(rUz)
    =e^{-a(f(z)-\log r)}
    =r^a\xi(z).
\end{equation}
Taking the complex Hessian of the two sides and using
\eqref{eq:strict-xi-Hessian-lambda},
\begin{align*}
\frac{\partial^2(\xi\circ\gamma)}
 {\partial z_i\,\partial\bar z_j}
&=
r^2\sum_{p,q}
U_{pi}\overline{U_{qj}}\,
\xi_{p\bar q}(rUz)
 =r^2\lambda\sum_pU_{pi}\overline{U_{pj}}
 =r^2\lambda\delta_{ij},\\
\frac{\partial^2(r^a\xi)}
 {\partial z_i\,\partial\bar z_j}
&=r^a\lambda\delta_{ij}.
\end{align*}
Therefore
\begin{equation}\label{eq:r-a-lambda}
    (r^2-r^a)\lambda=0.
\end{equation}

If \(\lambda=0\), then \(\xi\) is positive and real pluriharmonic.
Lemma~\ref{lem:positive-pluriharmonic} gives \(\xi\equiv C\) for some
\(C>0\).  Equation \eqref{eq:strict-xi-automorphy} would then give
\(C=r^aC\), which is impossible because \(0<r<1\) and \(a>0\).
Therefore
\begin{equation}\label{eq:lambda-nonzero}
    \lambda\neq0.
\end{equation}
Equation \eqref{eq:r-a-lambda} now gives \(r^2=r^a\).  Since
\(0<r<1\), it follows that \(a=2\).  Recalling
\(a=(1-t)^2/2\), we obtain
\begin{equation}\label{eq:exceptional-t-proof}
    \frac{(1-t)^2}{2}=2
    \quad\Longleftrightarrow\quad
    t=-1\ \text{or}\ t=3.
\end{equation}

\subsection{The Hopf cover and the admissible metric}
\label{subsec:hopf}

We retain the strict LCK hypotheses.  By the previous subsection,
\[
    a=2,
    \qquad
    \xi=e^{-2f},
    \qquad
    \xi_{i\bar j}=\lambda\delta_{ij},
    \qquad
    \lambda\neq0.
\]

\subsubsection{The cyclic quotient gives a holomorphic covering}

Let \(\gamma(z)=rUz\) be the chosen contraction.  The cyclic
quotient
\begin{equation}\label{eq:cyclic-quotient}
    M_\gamma
    :=
    (\C^n\setminus\{0\})/\langle\gamma\rangle
\end{equation}
is compact.  Indeed, for every \(z\neq0\), one can choose
\(k\in\mathbb Z\) such that
\[
    r\leq r^k|z|\leq1.
\]
Since \(|\gamma^kz|=r^k|z|\), the closed annulus
\[
    \{z\in\C^n:r\leq|z|\leq1\}
\]
meets every \(\langle\gamma\rangle\)-orbit.  Its image under the
quotient map is all of \(M_\gamma\).  Thus \(M_\gamma\) is the
continuous image of a compact set.  Since
\(\langle\gamma\rangle\subset\Deck(\pi)\), the subgroup--covering
correspondence gives a holomorphic covering
\[
    p:M_\gamma\longrightarrow M.
\]
Let
\begin{equation}\label{eq:q-quotient-map}
    q:\C^n\setminus\{0\}\longrightarrow M_\gamma
\end{equation}
be the quotient map.  Under the Euclidean identification of the
universal cover, the maps satisfy
\begin{equation}\label{eq:pi-factorization}
    \pi=p\circ q.
\end{equation}

Since \(U\) is unitarily diagonalizable, after a unitary change of
coordinates write
\begin{equation}\label{eq:D-generator}
    \gamma(z)=Dz,
    \qquad
    D=\operatorname{diag}(d_1,\dots,d_n),
    \qquad |d_j|=r<1.
\end{equation}
Thus \(M_\gamma=M_D\) is an isosceles Hopf manifold.
Under this identification, \(q=q_D\).

\subsubsection{Determination of \texorpdfstring{\(\xi\)}{xi}}

Put \(u=\xi-\lambda|z|^2\).  Since \(u_{i\bar j}=0\), the form
\(\partial u\) is closed and holomorphic.  The domain
\(\C^n\setminus\{0\}\) is simply connected, so the Poincar\'e lemma gives
a holomorphic primitive \(\varphi\) with \(d\varphi=\partial u\).
Since \(d(u-\varphi-\overline\varphi)=0\), after adding a real constant
to \(\varphi\), one has
\(u=\varphi+\overline{\varphi}\).  Hence
\begin{equation}\label{eq:xi-phi-general}
    \xi
    =
    \lambda|z|^2+\varphi(z)+\overline{\varphi(z)}.
\end{equation}
Hartogs' extension theorem extends \(\varphi\) across the origin.

Since \(a=2\) and \(\gamma(z)=Dz\), equation
\eqref{eq:strict-xi-automorphy} becomes
\begin{equation}\label{eq:xi-D-automorphy}
    \xi(Dz)=r^2\xi(z).
\end{equation}
Because \(|d_j|=r\), one has \(|Dz|^2=r^2|z|^2\).  Thus substitution
of \eqref{eq:xi-phi-general} into \eqref{eq:xi-D-automorphy} cancels the
\(\lambda|z|^2\)-terms and yields simply
\[
    \operatorname{Re}\bigl(\varphi(Dz)-r^2\varphi(z)\bigr)=0.
\]
The expression in parentheses is holomorphic and therefore, by the open
mapping theorem, is a purely imaginary constant \(c\).  Replacing
\(\varphi\) by
\(\varphi-c/(1-r^2)\) does not change
\(\varphi+\overline\varphi\) and makes this constant zero.  Thus
\begin{equation}\label{eq:phi-homogeneity}
    \varphi(Dz)=r^2\varphi(z).
\end{equation}
Write the Taylor expansion of \(\varphi\) at the origin as
\[
 \varphi=\sum_{m=0}^\infty P_m,
\]
where \(P_m\) is homogeneous of degree \(m\).  Comparing the
homogeneous terms in \eqref{eq:phi-homogeneity}, we obtain
\[
 P_m(Uz)=r^{2-m}P_m(z).
\]
Since \(U\) preserves the unit sphere,
\[
 \max_{|z|=1}|P_m(Uz)|
 =
 \max_{|z|=1}|P_m(z)|.
\]
Thus, if \(P_m\neq0\), the preceding equation gives
\(r^{2-m}=1\).  Since \(0<r<1\), this is possible only when \(m=2\).
Therefore \(\varphi\) is a homogeneous quadratic polynomial.  Since
\(\lambda\neq0\), there is a complex symmetric matrix \(\mathsf A\)
such that
\begin{equation}\label{eq:phi-matrix}
 \varphi(z)=\lambda z^{\mathsf T}\mathsf A z,
 \qquad
 \mathsf A^{\mathsf T}=\mathsf A.
\end{equation}
Define
\[
 Q_{\mathsf A}(z)
 :=
 |z|^2+z^{\mathsf T}\mathsf A z
 +\overline{z^{\mathsf T}\mathsf A z}.
\]
Then \eqref{eq:xi-phi-general} becomes
\begin{equation}\label{eq:xi-admissible}
 \xi=\lambda Q_{\mathsf A}.
\end{equation}
The holomorphic and antiholomorphic quadratic terms in
\eqref{eq:xi-admissible} have zero average over the unit sphere
\(S^{2n-1}\).  Hence
\begin{equation}\label{eq:lambda-positive}
 \lambda
 =
 \frac{1}{\operatorname{Area}(S^{2n-1})}
 \int_{S^{2n-1}}\xi\,dS
 >0.
\end{equation}
Since \(\xi=\lambda Q_{\mathsf A}\), we also have
\[
 Q_{\mathsf A}(z)>0
 \qquad (z\neq0).
\]

Substituting \eqref{eq:phi-matrix} into
\eqref{eq:phi-homogeneity} and using \(\lambda\neq0\), we obtain
\begin{equation}\label{eq:A-equivariance}
 D^{\mathsf T}\mathsf A D=r^2\mathsf A.
\end{equation}
By the Autonne--Takagi factorization, there is a unitary matrix \(V\)
such that
\[
 \mathsf A
 =
 V^{\mathsf T}\operatorname{diag}(s_1,\dots,s_n)V,
 \qquad s_j\geq0.
\]
Put \(w=Vz\) and write \(w_j=x_j+\sqrt{-1}y_j\).  Then
\[
 Q_{\mathsf A}(z)
 =
 \sum_{j=1}^n
 \bigl((1+2s_j)x_j^2+(1-2s_j)y_j^2\bigr).
\]
Since \(Q_{\mathsf A}>0\) away from the origin, \(s_j<1/2\) for every
\(j\).  The eigenvalues of
\(\mathsf A\overline{\mathsf A}\) are \(s_1^2,\dots,s_n^2\).  Hence
\begin{equation}\label{eq:A-positivity}
 \mathsf A\overline{\mathsf A}<\frac14I.
\end{equation}

Since \(a=2\), one has \(\xi=e^{-2f}\).  Put
\(c_0:=1/\lambda>0\).  Using \(\pi=p\circ q\) and
\(\xi=\lambda Q_{\mathsf A}\), we obtain
\begin{equation}\label{eq:metric-admissible-proof}
 q^*(p^*h)
 =
 \pi^*h
 =
 \xi^{-1}g_0
 =
 \frac{c_0}{Q_{\mathsf A}(z)}g_0.
\end{equation}
Equations \eqref{eq:phi-matrix}, \eqref{eq:A-equivariance},
\eqref{eq:A-positivity}, and \eqref{eq:metric-admissible-proof} show
that \(p^*h\) is admissible.  Thus \(p:M_D\to M\) is the required
holomorphic covering.  Together with Section~\ref{sec:gck}, this proves
the pointwise constant part of
Theorem~\ref{thm:main-classification}.

Conversely, every admissible metric on \(M_D\) has pointwise constant
Gauduchon holomorphic sectional curvature for \(t=-1\) and \(t=3\)
\cite{ChenNie}.

\subsection{The globally constant Gauduchon case}
\label{subsec:global-constant}

Assume that
\[
 \widetilde\kappa\equiv c.
\]
If the Lee form is exact, Section~\ref{sec:gck} shows that \(h\) is
K\"ahler, and \eqref{eq:GCK-space-form-tensor} gives
\[
 R^h=\frac{c}{2}G_h.
\]
Thus \((M,h)\) is a complex space form of holomorphic sectional
curvature \(c\).

Assume now that the Lee form is not exact.  Then \(a=2\) and
\[
 \xi=\lambda Q_{\mathsf A},
 \qquad
 \lambda>0.
\]
Substituting this into \eqref{eq:strict-lambda-expanded} gives
\[
 \widetilde\kappa
 =
 \lambda\left(
 1-\frac{|\partial Q_{\mathsf A}|_{g_0}^2}{Q_{\mathsf A}}
 \right).
\]
Since
\[
 \partial_iQ_{\mathsf A}
 =
 \bar z_i+2(\mathsf A z)_i,
\]
a direct expansion gives
\begin{align}
 |\partial Q_{\mathsf A}|_{g_0}^2
 &=
 |z|^2
 +2z^{\mathsf T}\mathsf A z
 +2\overline{z^{\mathsf T}\mathsf A z}
 +4z^{\mathsf T}\mathsf A\overline{\mathsf A}\,\bar z,
 \notag\\
 \widetilde\kappa(z)
 &=
 -\frac{\lambda}{Q_{\mathsf A}(z)}
 \left(
 z^{\mathsf T}\mathsf A z
 +\overline{z^{\mathsf T}\mathsf A z}
 +4z^{\mathsf T}\mathsf A\overline{\mathsf A}\,\bar z
 \right).
\label{eq:admissible-curvature-intro}
\end{align}
Therefore \(\widetilde\kappa\equiv c\) is equivalent to
\[
 cQ_{\mathsf A}(z)
 +\lambda\left(
 z^{\mathsf T}\mathsf A z
 +\overline{z^{\mathsf T}\mathsf A z}
 +4z^{\mathsf T}\mathsf A\overline{\mathsf A}\,\bar z
 \right)
 =0.
\]
Replacing \(z\) by \(e^{\sqrt{-1}\vartheta}z\) and comparing the
coefficients of \(e^{2\sqrt{-1}\vartheta}\),
\(e^{-2\sqrt{-1}\vartheta}\), and \(1\), we obtain
\[
 (\lambda+c)\mathsf A=0,
 \qquad
 4\lambda\mathsf A\overline{\mathsf A}+cI=0.
\]
If \(\mathsf A\neq0\), the first identity gives \(c=-\lambda\), and
the second gives
\[
 \mathsf A\overline{\mathsf A}=\frac14I,
\]
contrary to \eqref{eq:A-positivity}.  Hence
\(\mathsf A=0\), and the second identity gives \(c=0\).

Finally, \(Q_{\mathsf A}(z)=|z|^2\), and
\eqref{eq:metric-admissible-proof} becomes
\[
 q^*(p^*h)
 =
 \frac{1}{\lambda|z|^2}g_0.
\]
Thus \(p^*h\) is the standard Hopf metric up to scaling.  This proves
the globally constant assertion in
Theorem~\ref{thm:main-classification}.

\section{The two-parameter canonical connections}
\label{sec:two-parameter}

The proof is similar to the Gauduchon case.  We only give the
necessary changes.  Set
\begin{equation}\label{eq:two-chi-a}
 \chi(t,s)
 :=
 \bigl(1-t(1-s)\bigr)^2+s^2
 =
 (1-t+ts)^2+s^2,
 \qquad
 a_{t,s}:=\frac{\chi(t,s)}2.
\end{equation}
Thus \(a_{t,s}\geq0\), and \(a_{t,s}=0\) if and only if
\((t,s)=(1,0)\).

Assume that \(H_h^D\) is pointwise constant with curvature function
\(\kappa\), and put \(\widetilde\kappa:=\pi^*\kappa\).

\subsection{The conformal identity}

Retain the notation \(\widetilde h=e^{2f}g\) from
Subsection~\ref{subsec:cover}.  Since \(g\) is K\"ahler, the conformal
formula for \(\mathcal D_s^t\) in \cite{ChenNie} is obtained from
\eqref{eq:CN410} by replacing
\[
 \widehat{\widetilde R}^{\,t}
 \quad\text{with}\quad
 \widehat{\widetilde R}^{\,D},
 \qquad
 (1-t)^2
 \quad\text{with}\quad
 \chi(t,s).
\]
Define
\begin{equation}\label{eq:two-A-beta-ts}
 \mathcal A_{i\bar j}
 :=
 f_{i\bar j}-a_{t,s}f_i f_{\bar j},
 \qquad
 \beta_{t,s}
 :=
 e^{2f}\widetilde\kappa
 +2a_{t,s}|\partial f|_g^2.
\end{equation}
By polarization, the conformal formula gives
\begin{equation}\label{eq:two-curvature-identity}
 R^g
 =
 \frac{\beta_{t,s}}2G_g
 +\frac12L_g(\mathcal A).
\end{equation}
The K\"ahler curvature decomposition and the trace-free Ricci
contraction give
\begin{equation}\label{eq:two-bochner-ricci}
 \Boch(g)=0,
 \qquad
 (\Ric^g)^\circ
 =
 \frac{n+2}{2}\mathcal A^\circ.
\end{equation}
It follows from \eqref{eq:TV-equivalence} that
\(B^{\mathrm{TV}}(h)=0\).

\subsection{Proof of the classification}

If the Lee form is exact, the argument in
Section~\ref{sec:gck} applies with
\[
 a,\quad A,\quad\beta_t
 \qquad\text{replaced by}\qquad
 a_{t,s},\quad\mathcal A,\quad\beta_{t,s}.
\]
It shows that the conformal function \(f\) is constant.  Hence \(h\)
is K\"ahler.  The last part of Section~\ref{sec:gck} also shows that
\(\kappa\) is constant and that \((M,h)\) is a complex space form.

Assume now that the Lee form is not exact.  By
\eqref{eq:two-bochner-ricci}, \eqref{eq:TV-equivalence}, and
Proposition~\ref{prop:uniformization}, we may use the Euclidean cover
\[
 (\widetilde M,g)
 \cong
 (\C^n\setminus\{0\},g_0)
\]
and choose a deck contraction
\[
 \gamma(z)=rUz,
 \qquad
 0<r<1,
 \qquad
 U\in U(n),
\]
such that
\[
 f(rUz)=f(z)-\log r.
\]
Since \(R^{g_0}=0\), equation
\eqref{eq:two-curvature-identity} and
Lemma~\ref{lem:L-injective} give
\begin{equation}\label{eq:two-A-flat}
 \mathcal A
 =
 -\frac{\beta_{t,s}}2g_0.
\end{equation}

From this point, the argument in
Subsections~\ref{subsec:exclude-chern}--\ref{subsec:hopf} applies with
\[
 a,\qquad A,\qquad\beta_t
 \quad\text{replaced by}\quad
 a_{t,s},\qquad\mathcal A,\qquad\beta_{t,s}.
\]
We indicate the step that determines the parameters.  If
\(a_{t,s}=0\), then \((t,s)=(1,0)\), and
Subsection~\ref{subsec:exclude-chern} gives a contradiction.  Hence
\(a_{t,s}>0\).

Put
\[
 \xi:=e^{-a_{t,s}f}.
\]
Equation \eqref{eq:two-A-flat} gives
\[
 \xi_{i\bar j}
 =
 \lambda\delta_{ij},
 \qquad
 \lambda
 :=
 \frac{a_{t,s}\beta_{t,s}}2\xi.
\]
By Lemma~\ref{lem:constant-lambda}, \(\lambda\) is constant.  Moreover,
\[
 \xi(rUz)=r^{a_{t,s}}\xi(z).
\]
The same positive pluriharmonic argument as in
Subsection~\ref{subsec:flat} shows that \(\lambda\neq0\).  Taking the
complex Hessian of the last identity gives
\[
 \bigl(r^2-r^{a_{t,s}}\bigr)\lambda=0.
\]
Since \(0<r<1\), we obtain
\begin{equation}\label{eq:two-exceptional-proof}
 a_{t,s}=2,
 \qquad
 \chi(t,s)=4.
\end{equation}
For later use, the definition of \(\beta_{t,s}\) also gives
\begin{equation}\label{eq:two-lambda-expanded}
 \lambda
 =
 \frac{a_{t,s}}2\widetilde\kappa\,
 \xi^{\,1-2/a_{t,s}}
 +\frac{|\partial\xi|_{g_0}^2}{\xi}.
\end{equation}

When \(a_{t,s}=2\), the equations for \(\xi\) are exactly those used
in Subsection~\ref{subsec:hopf}.  The argument there gives a
holomorphic covering \(p:M_D\to M\), a constant \(\lambda>0\), and a
complex matrix \(\mathsf A\) satisfying
\[
 \xi=\lambda Q_{\mathsf A},
 \qquad
 \mathsf A^{\mathsf T}=\mathsf A,
 \qquad
 \mathsf A\overline{\mathsf A}<\frac14I,
 \qquad
 D^{\mathsf T}\mathsf A D=r^2\mathsf A.
\]
Furthermore,
\begin{equation}\label{eq:two-admissible-metric}
 q_D^*(p^*h)
 =
 \frac{1/\lambda}{Q_{\mathsf A}(z)}g_0.
\end{equation}
Thus \(p^*h\) is admissible.  This proves the pointwise constant part
of Theorem~\ref{thm:two-parameter-classification}.

Conversely, if \(\chi(t,s)=4\), every admissible metric on \(M_D\)
has pointwise constant \(\mathcal D_s^t\)-holomorphic sectional
curvature \cite{ChenNie}.

\subsection{The globally constant two-parameter case}

Suppose that \(\kappa\equiv c\).  The exact Lee-form case was settled
above.  If the Lee form is not exact, then \(a_{t,s}=2\), and
\eqref{eq:two-lambda-expanded} becomes
\[
 c
 =
 \lambda-\frac{|\partial\xi|_{g_0}^2}{\xi}.
\]
Substituting \(\xi=\lambda Q_{\mathsf A}\) and using
\eqref{eq:admissible-curvature-intro}, the polynomial comparison in
Subsection~\ref{subsec:global-constant} gives
\[
 (\lambda+c)\mathsf A=0,
 \qquad
 4\lambda\mathsf A\overline{\mathsf A}+cI=0.
\]
If \(\mathsf A\neq0\), these identities give
\[
 \mathsf A\overline{\mathsf A}=\frac14I,
\]
contrary to the admissibility condition.  Hence
\(\mathsf A=0\) and \(c=0\).  Equation
\eqref{eq:two-admissible-metric} becomes
\[
 q_D^*(p^*h)
 =
 \frac{1}{\lambda|z|^2}g_0.
\]
Thus \(p^*h\) is the standard Hopf metric up to scaling.  This proves
the globally constant assertion in
Theorem~\ref{thm:two-parameter-classification}.

\section{Chern curvature on balanced threefolds}
\label{sec:balanced-curvature}

\subsection{Chern torsion and Chern--Ricci contractions}

Throughout Sections~\ref{sec:balanced-curvature} and
\ref{sec:balanced-torsion}, \(g\) denotes the balanced Hermitian
metric, not the K\"ahler metric on the LCK cover used above.  We put
\(\nabla=\nabla^c\).  Thus \(R^c\) and \(T\) denote the Chern curvature
and torsion of \(g\), as in Section~\ref{sec:preliminaries}.
Let $\{e_i\}_{i=1}^n$ be a local \(g\)-unitary frame of
$T^{1,0}M$, with dual coframe $\{\varphi^i\}$.  On $T^{1,0}M$, set
\[
  \ip{X}{Y}=g(X,\overline Y).
\]
This is complex linear in the first argument and conjugate linear in the second; the same notation will be used for the induced Hermitian contractions on tensor bundles.  We use
\[
  \omega=\ii\sum_i\varphi^i\wedge\overline{\varphi^i}.
\]
We write $\dd V=\omega^n/n!$.  The operator $\Lambda_\omega$ is the metric contraction adjoint to exterior multiplication by $\omega$, and $*$ denotes the Hodge star of $g$.

For a tensor \(A\), we write \(\nabla_{e_s}A\) and
\(\nabla_{\bar e_s}A\) for its covariant derivatives as a tensor.
When local components are displayed, a comma followed by an index
denotes Chern covariant differentiation.  In particular,
\[
  T^j_{ik,s}:=(\nabla_{e_s}T)^j_{ik},
  \qquad
  T^j_{ik,\bar s}:=(\nabla_{\bar e_s}T)^j_{ik}.
\]
The same comma convention will be used for the components of other
tensors.
For tensors \(A\) and \(B\) of the same type, metric compatibility
and conjugate linearity in the second slot give
\begin{equation}\label{eq:complex-metric-compatibility}
  e_s\ip{A}{B}
  =\ip{\nabla_{e_s}A}{B}+\ip{A}{\nabla_{\bar e_s}B}.
\end{equation}

We also write
\[
  \nabla^{1,0}T=\sum_s\varphi^s\otimes\nabla_{e_s}T,
  \qquad
  \nabla^{0,1}T
  =\sum_s\overline{\varphi^s}\otimes\nabla_{\bar e_s}T.
\]
Their ordered-index norms are
\begin{equation}\label{eq:norm-conventions}
\begin{aligned}
  |T|^2&=\sum_{i,k,j}|T^j_{ik}|^2,\\
  |\nabla^{0,1}T|^2&=\sum_{s,i,k,j}|T^j_{ik,\bar s}|^2,
  \qquad
  |\nabla^{1,0}T|^2=\sum_{s,i,k,j}|T^j_{ik,s}|^2.
\end{aligned}
\end{equation}
The torsion one-form has components
\(\eta_i=\sum_sT^s_{si}\).  The metric is balanced precisely when
\(\eta=0\) \cite{GauduchonTorsion,ZhouZheng}.  Hence
\begin{equation}\label{eq:balanced-trace}
  \sum_sT^s_{si}=\sum_sT^s_{is}=0.
\end{equation}

The Chern--Ricci coefficient tensors used below are
\begin{equation}\label{eq:ricci}
  \rhoone_{i\bar j}=\sum_kR^c_{i\bar j k\bar k},
  \qquad
  \rhotwo_{i\bar j}=\sum_kR^c_{k\bar k i\bar j},
  \qquad
  \rhothree_{i\bar j}=\sum_kR^c_{i\bar k k\bar j}.
\end{equation}
For a complex \((1,1)\)-tensor \(H\), write
\[
 H_{i\bar j}:=H(e_i,\bar e_j).
\]
We use the same symbol \(H\) for the associated complex
\((1,1)\)-form
\[
 \sqrt{-1}\sum_{i,j=1}^n
 H_{i\bar j}\varphi^i\wedge\overline{\varphi^j}.
\]
This form is real if and only if
\[
 H_{i\bar j}=\overline{H_{j\bar i}},
\]
that is, if and only if \(H\) is Hermitian.  The forms
\(\rhoone\) and \(\rhotwo\) are real, whereas
\(\rhothree\) need not be real.  For any such tensor \(H\), define
\[
 (H^*)_{i\bar j}:=\overline{H_{j\bar i}},
 \qquad
 \operatorname{Re}H:=\frac12(H+H^*),
 \qquad
 \tr_\omega H:=\sum_{i=1}^nH_{i\bar i}.
\]
Put
\[
 s^{(1)}:=\tr_\omega\rhoone,
 \qquad
 s^{(3)}:=\tr_\omega\rhothree.
\]
With our curvature convention, the first Chern--Bianchi identities in
\cite{ChenChenNie} take the form
\begin{align}
 R^c_{i\bar j k\bar\ell}-R^c_{k\bar j i\bar\ell}
 &=-T^\ell_{ik,\bar j},
 \label{eq:bianchi-one}\\
 R^c_{i\bar j k\bar\ell}-R^c_{i\bar\ell k\bar j}
 &=-\overline{T^k_{j\ell,\bar i}}.
 \label{eq:bianchi-two}
\end{align}  Contracting
\eqref{eq:bianchi-two} in \(k,\ell\) gives
\begin{equation}\label{eq:rho13-direct}
  (\rhoone-\rhothree)_{i\bar j}
  =-\sum_k\overline{T^k_{jk,\bar i}}=0,
\end{equation}
because $\sum_kT^k_{jk}=-\sum_kT^k_{kj}=0$ is the balanced torsion-trace identity and its covariant derivative also vanishes.  Therefore
\begin{equation}\label{eq:balanced-rho3}
  \rhothree=\rhoone,
  \qquad s^{(1)}=s^{(3)}.
\end{equation}

Assume now that \(H_g^c\equiv c\).  By
\eqref{eq:pointwise-hsc-h}, this is equivalent to
\begin{equation}\label{eq:polarization}
 R^c_{i\bar j k\bar\ell}+R^c_{k\bar j i\bar\ell}
 +R^c_{i\bar\ell k\bar j}+R^c_{k\bar\ell i\bar j}
 =2c(\delta_{ij}\delta_{k\ell}+\delta_{i\ell}\delta_{kj}).
\end{equation}
Contracting \(k\) and \(\ell\) in \eqref{eq:polarization}, and then
taking the trace, gives
\begin{equation}\label{eq:polarized-contractions}
 \rhoone+\rhotwo+2\operatorname{Re}\rhothree
 =2(n+1)c\omega,
 \qquad
 s^{(1)}+s^{(3)}=n(n+1)c.
\end{equation}
Here \(s^{(3)}\) is real, since the Hermitian symmetry of \(R^c\)
gives
\[
 \overline{s^{(3)}}
 =
 \sum_{i,k=1}^n\overline{R^c_{i\bar k k\bar i}}
 =
 \sum_{i,k=1}^nR^c_{k\bar i i\bar k}
 =
 s^{(3)}.
\]

From now on, \(n=3\).  Equations \eqref{eq:balanced-rho3} and
\eqref{eq:polarized-contractions} give
\begin{equation}\label{eq:balanced-contractions}
  s^{(1)}=s^{(3)}=6c,
  \qquad
  3\rhoone+\rhotwo=8c\omega.
\end{equation}
Define
\begin{equation}\label{eq:sigma}
  \sigma=\rhoone-2c\omega.
\end{equation}
We use \( |\sigma|^2=\sum_{i,j}|\sigma_{i\bar j}|^2\).
Then $\sigma$ is a primitive real $(1,1)$-form and
\begin{equation}\label{eq:rho-sigma}
  \Lambda_\omega\sigma=0,
  \qquad
  \rhoone=2c\omega+\sigma,
  \qquad
  \rhotwo=2c\omega-3\sigma.
\end{equation}

\subsection{Curvature formulas in complex dimension three}
\label{subsec:balanced-reconstruction}

Following \cite{MaNie}, let \(N\) denote the tensor given by
\[
 N(X,Y,\bar Z,\bar W)
 =
 \ip{(\nabla_{\bar Z}T)(X,Y)}{W}
 -
 \ip{(\nabla_{\bar W}T)(X,Y)}{Z},
\]
where \(X,Y,Z,W\in T^{1,0}M\).  In a local \(g\)-unitary frame
\(\{e_i\}_{i=1}^3\), its components are
\begin{equation}\label{eq:N}
 N_{ik\bar j\bar\ell}
 =
 T^\ell_{ik,\bar j}-T^j_{ik,\bar\ell}.
\end{equation}
Thus \(N\) is skew-symmetric in \(i,k\) and in \(j,\ell\).  We write
\[
 |N|^2
 =
 \sum_{i,k,j,\ell=1}^3
 |N_{ik\bar j\bar\ell}|^2.
\]

The polarization identity and the Chern--Bianchi identities give the
following formula.

\begin{lemma}\label{lem:obstruction}
For a balanced threefold with \(H_g^c\equiv c\), one has
\begin{equation}\label{eq:obstruction}
R^c_{i\bar j k\bar\ell}
+
R^c_{k\bar\ell i\bar j}
=
c\bigl(
\delta_{ij}\delta_{k\ell}
+
\delta_{i\ell}\delta_{kj}
\bigr)
-
\frac{1}{2}N_{ik\bar j\bar\ell}.
\end{equation}
\end{lemma}

\begin{proof}
By \eqref{eq:bianchi-one},
\[
R^c_{k\bar j i\bar\ell}
=
R^c_{i\bar j k\bar\ell}
+
T^\ell_{ik,\bar j},
\]
and
\[
R^c_{i\bar\ell k\bar j}
=
R^c_{k\bar\ell i\bar j}
-
T^j_{ik,\bar\ell}.
\]
Therefore,
\begin{align*}
&
R^c_{i\bar j k\bar\ell}
+
R^c_{k\bar j i\bar\ell}
+
R^c_{i\bar\ell k\bar j}
+
R^c_{k\bar\ell i\bar j}
\\
&=
2\bigl(
R^c_{i\bar j k\bar\ell}
+
R^c_{k\bar\ell i\bar j}
\bigr)
+
T^\ell_{ik,\bar j}
-
T^j_{ik,\bar\ell}
\\
&=
2\bigl(
R^c_{i\bar j k\bar\ell}
+
R^c_{k\bar\ell i\bar j}
\bigr)
+
N_{ik\bar j\bar\ell}.
\end{align*}
Comparing this with \eqref{eq:polarization} and dividing by \(2\)
gives \eqref{eq:obstruction}.
\end{proof}
Let
\begin{equation}\label{eq:balanced-Pi}
  \Pi_{i\bar j}=\sum_sT^j_{is,\bar s}.
\end{equation}
Contracting \eqref{eq:obstruction} in $k,\ell$ gives $\rhoone+\rhotwo$ on the left.  On the other hand, balancedness yields
\[
 \sum_sN_{is\bar j\bar s}
 =\sum_s\bigl(T^s_{is,\bar j}-T^j_{is,\bar s}\bigr)
 =-\sum_sT^j_{is,\bar s}
 =-\Pi_{i\bar j}.
\]
Thus
\[
  \rhoone+\rhotwo=4c\omega+\frac12\Pi.
\]
Comparison with \eqref{eq:rho-sigma} gives
\begin{equation}\label{eq:Pi-sigma}
  \Pi=-4\sigma.
\end{equation}

The next elementary fact is the point at which complex dimension three is essential.

\begin{proposition}\label{prop:N-reconstruction}
Let $(M^3,g)$ be a balanced threefold with
$H_g^c\equiv c$. Then
\begin{equation}\label{eq:N-sigma}
 N_{ik\bar j\bar\ell}
 =4\bigl(
 \delta_{ij}\sigma_{k\bar\ell}-\delta_{i\ell}\sigma_{k\bar j}
 -\delta_{kj}\sigma_{i\bar\ell}+\delta_{k\ell}\sigma_{i\bar j}
 \bigr).
\end{equation}
In particular,
\begin{equation}\label{eq:N-norm}
N=0
\quad\Longleftrightarrow\quad
\sigma=0,
\qquad
|N|^2=64|\sigma|^2.
\end{equation}
\end{proposition}

\begin{proof}
Fix a point and a unitary frame. Let $\varepsilon_{ijk}$ be the
alternating symbol with $\varepsilon_{123}=1$. Recall that
\begin{equation}\label{eq:epsilon-basic-contraction}
\sum_{a=1}^3
\varepsilon_{ika}\varepsilon_{apq}
=
\delta_{ip}\delta_{kq}
-
\delta_{iq}\delta_{kp}.
\end{equation}

Define
\begin{equation}\label{eq:M-inverse}
M_{a\bar b}
=
\frac14
\sum_{p,q,r,v=1}^3
\varepsilon_{apq}
\overline{\varepsilon_{brv}}
N_{pq\bar r\bar v}.
\end{equation}
Indeed, substituting \eqref{eq:M-inverse} and contracting the
$a$- and $b$-indices, we obtain
\begin{align*}
\sum_{a,b}
\varepsilon_{ika}\overline{\varepsilon_{j\ell b}}M_{a\bar b}
&=
\frac14
\sum_{p,q,r,v}
\left(
\sum_a\varepsilon_{ika}\varepsilon_{apq}
\right)
\left(
\sum_b\overline{\varepsilon_{j\ell b}}
       \overline{\varepsilon_{brv}}
\right)
N_{pq\bar r\bar v} \\
&=
\frac14
\sum_{p,q,r,v}
(\delta_{ip}\delta_{kq}-\delta_{iq}\delta_{kp})
(\delta_{jr}\delta_{\ell v}-\delta_{jv}\delta_{\ell r})
N_{pq\bar r\bar v} \\
&=
N_{ik\bar j\bar\ell},
\end{align*}
where the last equality follows from the skew-symmetry of $N$ in
both index pairs. Hence
\begin{equation}\label{eq:N-M-explicit}
N_{ik\bar j\bar\ell}
=
\sum_{a,b=1}^3
\varepsilon_{ika}
\overline{\varepsilon_{j\ell b}}
M_{a\bar b}.
\end{equation}

Set
\[
C_{i\bar j}
=
\sum_s N_{is\bar s\bar j}.
\]
Since
\[
\sum_s
\varepsilon_{isa}
\overline{\varepsilon_{sjb}}
=
\delta_{ib}\delta_{aj}
-
\delta_{ij}\delta_{ab},
\]
Using \eqref{eq:N-M-explicit} and the preceding contraction identity,
we obtain
\begin{equation}\label{eq:C-M}
 C_{i\bar j}
 =\sum_{s,a,b=1}^3
  \varepsilon_{isa}\overline{\varepsilon_{sjb}}M_{a\bar b}
 =M_{j\bar i}-(\operatorname{tr}M)\delta_{ij}.
\end{equation}
Taking the trace gives
\[
 \operatorname{tr}C=-2\operatorname{tr}M.
\]
Therefore,
\begin{equation}\label{eq:M-from-C}
 M_{a\bar b}
 =
 C_{b\bar a}
 -
 \frac12(\operatorname{tr}C)\delta_{ab}.
\end{equation}

Since \(g\) is balanced,
\[
 \sum_{s=1}^3T^s_{is}=0.
\]
Taking the Chern covariant derivative in the \(\bar e_j\)-direction
gives
\[
 \sum_{s=1}^3T^s_{is,\bar j}=0.
\]
Hence, by \eqref{eq:N}, the definition of \(\Pi\), and
\eqref{eq:Pi-sigma},
\[
 C_{i\bar j}
 =\sum_{s=1}^3N_{is\bar s\bar j}
 =\sum_{s=1}^3\left(T^j_{is,\bar s}-T^s_{is,\bar j}\right)
 =\Pi_{i\bar j}
 =-4\sigma_{i\bar j}.
\]
Since \(\operatorname{tr}\sigma=0\), one also has
\(\operatorname{tr}C=0\).  Therefore, \eqref{eq:M-from-C} gives
\begin{equation}\label{eq:M-sigma}
 M_{a\bar b}=-4\sigma_{b\bar a}.
\end{equation}

Substituting \eqref{eq:M-sigma} into
\eqref{eq:N-M-explicit}, we obtain
\[
 N_{ik\bar j\bar\ell}
 =
 -4\sum_{a,b=1}^3
 \varepsilon_{ika}
 \overline{\varepsilon_{j\ell b}}
 \sigma_{b\bar a}.
\]
A direct computation, using \(\operatorname{tr}\sigma=0\), gives
\[
 \sum_{a,b=1}^3
 \varepsilon_{ika}\overline{\varepsilon_{j\ell b}}\sigma_{b\bar a}
 =-\delta_{ij}\sigma_{k\bar\ell}+\delta_{i\ell}\sigma_{k\bar j}
  +\delta_{kj}\sigma_{i\bar\ell}-\delta_{k\ell}\sigma_{i\bar j}.
\]
This proves \eqref{eq:N-sigma}.

Finally, the contraction identities
\[
 \sum_{i,k=1}^3
 \varepsilon_{ika}\overline{\varepsilon_{ikc}}
 =2\delta_{ac},
 \qquad
 \sum_{j,\ell=1}^3
 \overline{\varepsilon_{j\ell b}}\varepsilon_{j\ell d}
 =2\delta_{bd}
\]
and \eqref{eq:N-M-explicit} give
\[
 |N|^2=4|M|^2.
\]
By \eqref{eq:M-sigma}, we have
\[
 |M|^2=16|\sigma|^2.
\]
Thus
\[
 |N|^2=64|\sigma|^2,
\]
which proves \eqref{eq:N-norm}.
\end{proof}

We shall also use the following reconstruction formula.

\begin{lemma}\label{lem:full-reconstruction}
Let $(M^3,g)$ be a balanced threefold with
$H_g^c\equiv c$. Then
\begin{equation}\label{eq:full-reconstruction}
 2R^c_{i\bar j k\bar\ell}
 =c\bigl(\delta_{ij}\delta_{k\ell}+\delta_{i\ell}\delta_{kj}\bigr)
  -T^\ell_{ik,\bar j}-\overline{T^k_{j\ell,\bar i}}
  +\frac12N_{ik\bar j\bar\ell}.
\end{equation}
\end{lemma}

\begin{proof}
By \eqref{eq:bianchi-one}, \eqref{eq:bianchi-two}, and the definition of
$N$,
\[
R^c_{k\bar\ell i\bar j}
 =R^c_{i\bar\ell k\bar j}+T^j_{ik,\bar\ell}
 =R^c_{i\bar j k\bar\ell}
  +\overline{T^k_{j\ell,\bar i}}+T^\ell_{ik,\bar j}
  -N_{ik\bar j\bar\ell}.
\]
Substituting this into \eqref{eq:obstruction} gives
\[
2R^c_{i\bar j k\bar\ell}
=
c\bigl(
\delta_{ij}\delta_{k\ell}
+
\delta_{i\ell}\delta_{kj}
\bigr)
-
T^\ell_{ik,\bar j}
-
\overline{T^k_{j\ell,\bar i}}
+
\frac12N_{ik\bar j\bar\ell}.
\]
This proves \eqref{eq:full-reconstruction}.
\end{proof}

\section{Global torsion identities and the balanced classification}
\label{sec:balanced-torsion}

\subsection{The quadratic torsion term}

We recall the following pointwise torsion normal form, proved by
Zhou and Zheng \cite{ZhouZheng}.

\begin{lemma}\label{lem:torsion-frame}
Let $(M^3,g)$ be a balanced threefold. For each point
$p\in M$, there exists a unitary frame of $T^{1,0}_pM$ such that
\begin{equation}\label{eq:torsion-normal-form}
T^1_{23}=\tau_1,\qquad
T^2_{31}=\tau_2,\qquad
T^3_{12}=\tau_3,
\end{equation}
where $\tau_1,\tau_2,\tau_3\in\mathbb C$. By the skew-symmetry of
the Chern torsion,
\[
T^1_{32}=-\tau_1,\qquad
T^2_{13}=-\tau_2,\qquad
T^3_{21}=-\tau_3.
\]
All remaining torsion components vanish.
\end{lemma}
Following \cite{LiuYang}, define the Hermitian tensor
\begin{equation}\label{eq:Qt}
 \mathcal Q_{i\bar j}
 =
 \sum_{p,k=1}^3
 T^k_{ip}\overline{T^k_{jp}}.
\end{equation}
Its trace is
\begin{equation}\label{eq:Qt-trace}
 \Lambda_\omega\mathcal Q=|T|^2.
\end{equation}
Indeed, tracing \eqref{eq:Qt} gives
$\sum_{i,p,k}|T^k_{ip}|^2$, which is exactly the ordered-index norm fixed above.
Put
\[
  \Phi=\ii\partial\bar\partial\omega.
\]
Let \(\partial^*\) and \(\bar\partial^*\) denote the formal adjoints.
By \cite{LiuYang},
\[
  \rhotwo=\rhoone-\Lambda_\omega\Phi
  -(\partial\partial^*\omega+\bar\partial\bar\partial^*\omega)+\mathcal Q.
\]
Here the term written there as
$\ii\Lambda_\omega(\partial\bar\partial\omega)$ is
$\Lambda_\omega\Phi$ in our notation.  For a balanced metric,
$\partial^*\omega=\bar\partial^*\omega=0$, and hence
\begin{equation}\label{eq:ricci-comparison}
  \rhotwo=\rhoone-\Lambda_\omega\Phi+\mathcal Q.
\end{equation}
Assume for the rest of this subsection that \(H_g^c\equiv c\).
Together with \eqref{eq:rho-sigma}, \eqref{eq:ricci-comparison} gives
\begin{equation}\label{eq:Phi-contractions}
  \Lambda_\omega\Phi=4\sigma+\mathcal Q,
  \qquad
  \Lambda_\omega^2\Phi=|T|^2.
\end{equation}
Indeed, applying $\Lambda_\omega$ to the first equality and using
$\Lambda_\omega\sigma=0$ together with \eqref{eq:Qt-trace} gives the second.

By the Hodge star formula and the primitive decomposition
\cite{Demailly}, every real
$(2,2)$-form $\Psi$ on a Hermitian threefold satisfies
\begin{equation}\label{eq:star-22}
*\Psi
=
-\Lambda_\omega\Psi
+
\frac12\bigl(\Lambda_\omega^2\Psi\bigr)\omega.
\end{equation}
Combining \eqref{eq:star-22} with the contraction identities
\eqref{eq:Phi-contractions}, we obtain
\[
 *\Phi
 =-\Lambda_\omega\Phi+\frac12\bigl(\Lambda_\omega^2\Phi\bigr)\omega
 =-\bigl(4\sigma+\mathcal Q\bigr)+\frac12|T|^2\omega.
\]
Hence
\begin{equation}\label{eq:star-Phi}
*\Phi
=
-4\sigma-\mathcal Q+\frac12|T|^2\omega.
\end{equation}

Assume in addition, for the rest of this subsection, that \(M\) is
compact.  By \cite{FusiGiusti},
\begin{equation}\label{eq:fusi-star}
  \partial^*\partial\omega
  =*\left(\frac{\ii\partial\bar\partial\omega^{n-2}}{(n-2)!}\right)
\end{equation}
In dimension three this becomes
\begin{equation}\label{eq:fusi-star-dim3}
  \partial^*\partial\omega=*\Phi,
  \qquad \Phi=\ii\partial\bar\partial\omega.
\end{equation}

The orthogonality needed below follows from the closedness of the
first Chern--Ricci form.  In local holomorphic coordinates
\((z^1,\ldots,z^n)\), write
\[
 g_{i\bar j}
 =g\left(\frac{\partial}{\partial z^i},
         \frac{\partial}{\partial\bar z^j}\right),
\]
and let \((g^{i\bar j})\) be the inverse matrix of
\((g_{i\bar j})\).  Our curvature convention gives
\[
  \rhoone_{i\bar j}
  =-\partial_{\bar j}\bigl(g^{k\bar q}\partial_i g_{k\bar q}\bigr)
  =-\partial_i\partial_{\bar j}\log\det(g_{a\bar b}),
\]
and hence
\[
  \rhoone=-\ii\partial\bar\partial\log\det(g_{a\bar b}).
\]
Thus \(\dd\rhoone=0\), so \(\partial\rhoone=0\).  The definition of
the formal adjoint now gives
\begin{equation}\label{eq:ricci-orthogonality}
 \int_M\ip{\rhoone}{\partial^*\partial\omega}\,\dd V
 =\int_M\ip{\partial\rhoone}{\partial\omega}\,\dd V
 =0.
\end{equation}
At a point, in a unitary frame,
\[
  |\omega|^2=3,
  \qquad \ip{\sigma}{\omega}=0,
  \qquad \ip{\mathcal Q}{\omega}=\tr \mathcal Q=|T|^2.
\]
Both $\sigma$ and $\mathcal Q$ are Hermitian, so their pointwise inner product
$\ip{\sigma}{\mathcal Q}$ is real.
Therefore \eqref{eq:rho-sigma} and \eqref{eq:star-Phi} give
\[
 \ip{\rhoone}{*\Phi}
 =\ip{2c\omega+\sigma}{-4\sigma-\mathcal Q+\tfrac12|T|^2\omega}
 =c|T|^2-4|\sigma|^2-\ip{\sigma}{\mathcal Q}.
\]
Integrating the preceding identity and using
\eqref{eq:fusi-star-dim3} and
\eqref{eq:ricci-orthogonality}, we obtain
\begin{equation}\label{eq:sigma-Qt}
 \int_M\ip{\sigma}{\mathcal Q}\,\dd V
 =
 c\norm{T}_{L^2}^2
 -
 4\norm{\sigma}_{L^2}^2.
\end{equation}

\subsection{Chern Laplacians and the curvature action}

We use the following formal-adjoint identity repeatedly.
\begin{lemma}\label{lem:balanced-ibp}
Let $(M^n,g)$ be a compact balanced Hermitian manifold. Let
\[
A=\sum_s A_{\bar s}\,\overline{\varphi^s}
\]
be a tensor-valued $(0,1)$-form, and let $B$ be a tensor field of the
same coefficient type as $A_{\bar s}$. Then
\begin{equation}\label{eq:balanced-ibp}
\int_M \operatorname{Re}\sum_s
\left\langle
(\nabla_{e_s}A)(\bar e_s),B
\right\rangle\,\dd V
=
-\int_M \operatorname{Re}\sum_s
\left\langle
A_{\bar s},\nabla_{\bar e_s}B
\right\rangle\,\dd V.
\end{equation}

Similarly, if
\[
A=\sum_s A_s\varphi^s
\]
is a tensor-valued $(1,0)$-form and $B$ has the same coefficient type
as $A_s$, then
\begin{equation}\label{eq:balanced-ibp-conjugate}
\int_M \operatorname{Re}\sum_s
\left\langle
(\nabla_{\bar e_s}A)(e_s),B
\right\rangle\,\dd V
=
-\int_M \operatorname{Re}\sum_s
\left\langle
A_s,\nabla_{e_s}B
\right\rangle\,\dd V.
\end{equation}
\end{lemma}

\begin{proof}
Let
\[
X=\sum_s X^s e_s
\]
be a $(1,0)$ vector field. The Chern divergence formula gives
\[
\partial(\iota_X\dd V)
=
\left(
\sum_s
\left\langle
\nabla_{e_s}X,e_s
\right\rangle
-\eta(X)
\right)\dd V,
\]
where $\eta$ is the torsion one-form. Since $g$ is balanced,
$\eta=0$. Thus, by Stokes' theorem,
\begin{equation}\label{eq:balanced-divergence}
\int_M
\sum_s
\left\langle
\nabla_{e_s}X,e_s
\right\rangle
\,\dd V
=
0.
\end{equation}

Now define
\[
X
=
\sum_s
\left\langle
A_{\bar s},B
\right\rangle e_s.
\]
This definition is independent of the choice of unitary frame. Fix
$p\in M$ and choose a local unitary frame that is Chern-normal at $p$.
Then, at $p$,
\begin{align*}
\sum_s
\left\langle
\nabla_{e_s}X,e_s
\right\rangle
&=
\sum_s
e_s\left\langle A_{\bar s},B\right\rangle \\
&=
\sum_s
\left\langle
(\nabla_{e_s}A)(\bar e_s),B
\right\rangle
+
\sum_s
\left\langle
A_{\bar s},\nabla_{\bar e_s}B
\right\rangle.
\end{align*}
Since $p$ is arbitrary, this identity holds pointwise on $M$.
Taking real parts, integrating, and using
\eqref{eq:balanced-divergence} gives \eqref{eq:balanced-ibp}.

The second identity follows by applying the conjugate divergence
formula to
\[
X
=
\sum_s
\left\langle
A_s,B
\right\rangle \bar e_s.
\]
\end{proof}

For tensors $A$ and $B$ of the same type as $T$, we use the
ordered-index pointwise Hermitian inner product
\begin{equation}\label{eq:tensor-inner-product}
\ip{A}{B}
=
\sum_{i,k,j}
A^j_{ik}\overline{B^j_{ik}}.
\end{equation}

For tensorial second derivatives, set
\[
\nabla^2_{X,Y}T
:=
\nabla_X(\nabla_YT)-\nabla_{\nabla_XY}T.
\]
The corresponding component convention is
\[
  T^j_{ik,\bar\ell m}
  :=(\nabla^2_{e_m,\bar e_\ell}T)^j_{ik},
  \qquad
  T^j_{ik,m\bar\ell}
  :=(\nabla^2_{\bar e_\ell,e_m}T)^j_{ik}.
\]
Thus the comma indices record the order of differentiation from left
to right; the rightmost index is the outer derivative.
With respect to a local unitary frame $\{e_s\}$, define
\begin{equation}\label{eq:chern-laplacians-tensor}
\mathcal U
:=
\sum_s\nabla^2_{e_s,\bar e_s}T,
\qquad
\mathcal V
:=
\sum_s\nabla^2_{\bar e_s,e_s}T.
\end{equation}
Write
\[
\mathcal U
=
\frac12\sum_{i,k,j}
\mathcal U^j_{ik}\,
\varphi^i\wedge\varphi^k\otimes e_j,
\qquad
\mathcal V
=
\frac12\sum_{i,k,j}
\mathcal V^j_{ik}\,
\varphi^i\wedge\varphi^k\otimes e_j.
\]
Thus
\begin{equation}\label{eq:chern-laplacians}
\mathcal U^j_{ik}
=
\sum_s T^j_{ik,\bar s s},
\qquad
\mathcal V^j_{ik}
=
\sum_s T^j_{ik,s\bar s}.
\end{equation}

Since the Chern torsion has no mixed component, the tensorial
second-derivative difference is the usual mixed curvature commutator.
Hence
\begin{equation}\label{eq:UV-commutator}
  (\mathcal U-\mathcal V)^j_{ik}
  =
  \sum_s\left(
    T^j_{ik,\bar s s}-T^j_{ik,s\bar s}
  \right).
\end{equation}

The mixed Ricci commutation identity
\cite{ZhouZheng} is
\begin{equation}\label{eq:mixed-commutator}
 T^j_{ik,\bar\ell m}-T^j_{ik,m\bar\ell}
 =\sum_r\bigl(
 T^r_{ik}R^c_{m\bar\ell r\bar j}
 -T^j_{rk}R^c_{m\bar\ell i\bar r}
 -T^j_{ir}R^c_{m\bar\ell k\bar r}
 \bigr).
\end{equation}

For a Hermitian matrix $H=(H_{i\bar j})$, define
\begin{equation}\label{eq:CH}
\mathcal C_H(T)^j_{ik}
=
\sum_r\bigl(
T^r_{ik}H_{r\bar j}
-T^j_{rk}H_{i\bar r}
-T^j_{ir}H_{k\bar r}
\bigr).
\end{equation}
By \eqref{eq:tensor-inner-product},
\begin{equation}\label{eq:CH-pairing}
\ip{\mathcal C_H(T)}{T}
=
\sum_{i,k,j}
\mathcal C_H(T)^j_{ik}\,
\overline{T^j_{ik}}.
\end{equation}

Setting $m=\ell=s$ in \eqref{eq:mixed-commutator} and summing over
$s$, we obtain
\begin{equation}\label{eq:laplacian-commutator}
 (\mathcal U-\mathcal V)^j_{ik}
 =\sum_r\bigl(
 T^r_{ik}\rhotwo_{r\bar j}
 -T^j_{rk}\rhotwo_{i\bar r}
 -T^j_{ir}\rhotwo_{k\bar r}
 \bigr)
 =\mathcal C_{\rhotwo}(T)^j_{ik}.
\end{equation}

We first record the resulting algebraic contraction.

\begin{lemma}\label{lem:CH-contraction}
Let $(M^3,g)$ be a balanced threefold. For every Hermitian
matrix $H=(H_{i\bar j})$,
\begin{equation}\label{eq:CH-contraction}
\operatorname{Re}\ip{\mathcal C_H(T)}{T}
=
(\tr H)|T|^2-4\ip{H}{\mathcal Q}.
\end{equation}
If moreover $H_g^c\equiv c$, then
\begin{equation}\label{eq:laplacian-contraction}
\operatorname{Re}\ip{\mathcal U-\mathcal V}{T}
=
-2c|T|^2+12\ip{\sigma}{\mathcal Q}.
\end{equation}
\end{lemma}

\begin{proof}
Fix a point and choose the frame of Lemma~\ref{lem:torsion-frame}.
Then
\[
|T|^2
=
2\bigl(
|\tau_1|^2+|\tau_2|^2+|\tau_3|^2
\bigr),
\]
and \eqref{eq:Qt} gives
\begin{equation}\label{eq:Qt-tau-components}
\mathcal Q
=
\operatorname{diag}\bigl(
|\tau_2|^2+|\tau_3|^2,\,
|\tau_1|^2+|\tau_3|^2,\,
|\tau_1|^2+|\tau_2|^2
\bigr).
\end{equation}
In particular,
\[
\tr\mathcal Q=|T|^2.
\]

By \eqref{eq:CH},
\[
\mathcal C_H(T)^1_{23}
=
\bigl(
H_{1\bar1}-H_{2\bar2}-H_{3\bar3}
\bigr)\tau_1,
\]
\[
\mathcal C_H(T)^2_{31}
=
\bigl(
H_{2\bar2}-H_{3\bar3}-H_{1\bar1}
\bigr)\tau_2,
\]
and
\[
\mathcal C_H(T)^3_{12}
=
\bigl(
H_{3\bar3}-H_{1\bar1}-H_{2\bar2}
\bigr)\tau_3.
\]
The reversed components give the same contributions. Hence
\begin{align}
\operatorname{Re}\ip{\mathcal C_H(T)}{T}
={}&
2\bigl(
H_{1\bar1}-H_{2\bar2}-H_{3\bar3}
\bigr)|\tau_1|^2 \nonumber\\
&+
2\bigl(
H_{2\bar2}-H_{3\bar3}-H_{1\bar1}
\bigr)|\tau_2|^2 \nonumber\\
&+
2\bigl(
H_{3\bar3}-H_{1\bar1}-H_{2\bar2}
\bigr)|\tau_3|^2.
\label{eq:CH-component-sum}
\end{align}

Put
\[
S=|\tau_1|^2+|\tau_2|^2+|\tau_3|^2.
\]
By \eqref{eq:Qt-tau-components},
\[
\ip{H}{\mathcal Q}
=
(\tr H)S
-
\sum_{j=1}^3H_{j\bar j}|\tau_j|^2.
\]
Therefore,
\begin{align*}
(\tr H)|T|^2-4\ip{H}{\mathcal Q}
&=
2\sum_{j=1}^3
\bigl(
2H_{j\bar j}-\tr H
\bigr)|\tau_j|^2,
\end{align*}
which agrees with \eqref{eq:CH-component-sum}. This proves
\eqref{eq:CH-contraction}.

Now take
\[
H=\rhotwo=2c\omega-3\sigma.
\]
By \eqref{eq:laplacian-commutator},
\eqref{eq:balanced-contractions}, and \eqref{eq:Qt-trace},
\[
 \operatorname{Re}\ip{\mathcal U-\mathcal V}{T}
 =6c|T|^2-4\ip{2c\omega-3\sigma}{\mathcal Q}
 =-2c|T|^2+12\ip{\sigma}{\mathcal Q}.
\]
This proves \eqref{eq:laplacian-contraction}.
\end{proof}

We next relate \(\mathcal U\) to the curvature.  The vanishing torsion
trace also gives a holomorphic divergence identity.

\begin{lemma}\label{lem:holomorphic-divergence}
On a balanced Hermitian manifold, every local unitary frame
$\{e_s\}$ satisfies
\begin{equation}\label{eq:holomorphic-divergence}
\sum_s T^s_{ik,s}=0.
\end{equation}
\end{lemma}

\begin{proof}
Let $\mathfrak S$ denote the cyclic sum. With respect to a local
unitary frame $\{e_s\}$, the $(3,0)$ part of the first Bianchi identity
is
\[
\mathfrak S_{sik}
\bigl(
(\nabla_{e_s}T)(e_i,e_k)
+
T(T(e_s,e_i),e_k)
\bigr)
=
0.
\]
Taking the $e_s$-component and summing over $s$, we obtain
\begin{equation}\label{eq:pure-bianchi-expanded}
\begin{split}
0={}&
\sum_s T^s_{ik,s}
+\sum_s T^s_{ks,i}
+\sum_s T^s_{si,k}\\
&+\sum_{r,s}T^r_{si}T^s_{rk}
+\sum_{r,s}T^r_{ik}T^s_{rs}
+\sum_{r,s}T^r_{ks}T^s_{ri}.
\end{split}
\end{equation}

By balancedness,
\[
\sum_sT^s_{ks}
=
\sum_sT^s_{si}
=
\sum_sT^s_{rs}
=
0.
\]
Hence
\[
\sum_sT^s_{ks,i}
=
\sum_sT^s_{si,k}
=
\sum_{r,s}T^r_{ik}T^s_{rs}
=
0.
\]
Moreover,
\[
\sum_{r,s}T^r_{ks}T^s_{ri}
=
\sum_{r,s}T^s_{kr}T^r_{si}
=
-\sum_{r,s}T^s_{rk}T^r_{si}
=
-\sum_{r,s}T^r_{si}T^s_{rk}.
\]
Thus \eqref{eq:pure-bianchi-expanded} reduces to
\[
\sum_sT^s_{ik,s}=0.
\]
\end{proof}

Define the curvature action
\begin{equation}\label{eq:R-action}
\begin{split}
  \mathcal R(T)^j_{ik}=\sum_{s,r}\bigl(&
    T^r_{ik}R^c_{s\bar j r\bar s}
   -T^s_{rk}R^c_{s\bar j i\bar r}
   -T^s_{ir}R^c_{s\bar j k\bar r}\bigr)
\end{split}
\end{equation}
In a local unitary frame, define the \((1,0)\)-Chern divergence of
\(N\) by
\begin{equation}\label{eq:divN}
  (\DivCh N)^j_{ik}
  =\sum_s(\nabla_{e_s}N)_{ik\bar j\bar s}.
\end{equation}
To expand this derivative, choose a Chern-normal unitary frame at the
point under consideration. Then
\begin{equation}\label{eq:divN-expanded}
 (\DivCh N)^j_{ik}
 =\sum_s(\nabla_{e_s}N)_{ik\bar j\bar s}
 =\sum_s\bigl(T^s_{ik,\bar j s}-T^j_{ik,\bar s s}\bigr)
 =\sum_s T^s_{ik,\bar j s}-\mathcal U^j_{ik}.
\end{equation}
By Lemma~\ref{lem:holomorphic-divergence},
\[
\sum_s T^s_{ik,s}=0.
\]
Taking the $\bar e_j$-covariant derivative gives
\[
\sum_s T^s_{ik,s\bar j}=0.
\]
Hence
\begin{align*}
\sum_s T^s_{ik,\bar j s}
&=
\sum_s
\bigl(
T^s_{ik,\bar j s}
-
T^s_{ik,s\bar j}
\bigr).
\end{align*}
Applying \eqref{eq:mixed-commutator} with $m=s$, $\ell=j$, and upper
index $s$, and then summing over $s$, gives
\[
\begin{split}
  \sum_{s,r}\bigl(&T^r_{ik}R^c_{s\bar j r\bar s}
   -T^s_{rk}R^c_{s\bar j i\bar r}
   -T^s_{ir}R^c_{s\bar j k\bar r}\bigr)
  =\mathcal R(T)^j_{ik}.
\end{split}
\]
Thus \eqref{eq:divN-expanded} becomes
\begin{equation}\label{eq:holomorphic-laplacian-div}
  \mathcal U=\mathcal R(T)-\DivCh N.
\end{equation}

The following pointwise identity is the key algebraic step.

\begin{lemma}\label{lem:curvature-action}
Let $(M^3,g)$ be a balanced threefold with
$H_g^c\equiv c$. Then
\begin{equation}\label{eq:curvature-action}
\operatorname{Re}\ip{\mathcal R(T)}{T}
=
c|T|^2.
\end{equation}
\end{lemma}

\begin{proof}
Fix a point and choose the unitary frame of
Lemma~\ref{lem:torsion-frame}. Extend it locally to a unitary frame.
Let $i,j,k$ be pairwise distinct. By the torsion normal form, the
three terms in \eqref{eq:R-action} reduce to
\begin{align}
\sum_{s,r}T^r_{ik}R^c_{s\bar j r\bar s}
&=
T^j_{ik}\sum_sR^c_{s\bar j j\bar s},
\label{eq:R-action-I}\\
-\sum_{s,r}T^s_{rk}R^c_{s\bar j i\bar r}
&=
-T^j_{ik}R^c_{j\bar j i\bar i}
-T^i_{jk}R^c_{i\bar j i\bar j},
\label{eq:R-action-II}\\
-\sum_{s,r}T^s_{ir}R^c_{s\bar j k\bar r}
&=
-T^j_{ik}R^c_{j\bar j k\bar k}
-T^k_{ij}R^c_{k\bar j k\bar j}.
\label{eq:R-action-III}
\end{align}

For $m\neq j$, setting $i=k=m$ and $\ell=j$ in
\eqref{eq:polarization} gives
\begin{equation}\label{eq:isotropic-zero}
R^c_{m\bar j m\bar j}=0.
\end{equation}
Thus the last terms in \eqref{eq:R-action-II} and
\eqref{eq:R-action-III} vanish. Since
$\{i,j,k\}=\{1,2,3\}$ and
\[
R^c_{j\bar j j\bar j}=c,
\]
we obtain
\begin{equation}\label{eq:R-action-differences}
 \mathcal R(T)^j_{ik}
 =T^j_{ik}\bigl[
 c+R^c_{i\bar j j\bar i}-R^c_{j\bar j i\bar i}
  +R^c_{k\bar j j\bar k}-R^c_{j\bar j k\bar k}
 \bigr].
\end{equation}

For each $m\in\{i,k\}$, equation \eqref{eq:bianchi-one} gives
\[
R^c_{m\bar j j\bar m}
-
R^c_{j\bar j m\bar m}
=
-T^m_{mj,\bar j}.
\]
Moreover, balancedness gives
\[
\sum_rT^r_{rj}=0.
\]
Taking the $\bar e_j$-covariant derivative,
\[
T^i_{ij,\bar j}+T^k_{kj,\bar j}=0,
\]
since $T^j_{jj}=0$. Hence the two curvature differences in
\eqref{eq:R-action-differences} cancel, and
\[
\mathcal R(T)^j_{ik}
=
cT^j_{ik}.
\]

This holds for every torsion component that may be nonzero in the
chosen frame. The remaining components make no contribution to the
inner product. Therefore,
\[
\operatorname{Re}\ip{\mathcal R(T)}{T}
=
c|T|^2.
\]
\end{proof}

Finally, balanced integration by parts gives the exact contribution of $\DivCh N$.

\begin{lemma}\label{lem:N-divergence}
Let $(M^3,g)$ be a compact balanced threefold. Then
\begin{equation}\label{eq:N-divergence-general}
\int_M\operatorname{Re}\ip{\DivCh N}{T}\,\dd V
=
\frac12\norm{N}_{L^2}^2.
\end{equation}
If moreover \(H_g^c\equiv c\), then
\begin{equation}\label{eq:N-divergence}
  \frac12\norm{N}_{L^2}^2
  =
  32\norm{\sigma}_{L^2}^2.
\end{equation}
\end{lemma}

\begin{proof}
In a local unitary frame, let
\[
A
=
\frac12
\sum_{i,k,j,s}
N_{ik\bar j\bar s}\,
\overline{\varphi^s}
\otimes
\varphi^i\wedge\varphi^k\otimes e_j.
\]
Thus
\[
(A_{\bar s})^j_{ik}
=
N_{ik\bar j\bar s},
\qquad
(\DivCh N)^j_{ik}
=
\sum_s(\nabla_{e_s}N)_{ik\bar j\bar s}.
\]
Applying Lemma~\ref{lem:balanced-ibp} to $A$ and $T$ gives
\begin{equation}\label{eq:N-divergence-integration}
\int_M\operatorname{Re}\ip{\DivCh N}{T}\,\dd V
=
-\int_M\operatorname{Re}
\sum_{i,k,j,s}
N_{ik\bar j\bar s}
\overline{T^j_{ik,\bar s}}
\,\dd V.
\end{equation}

For fixed $i,k$,
\[
N_{ik\bar j\bar s}
=
T^s_{ik,\bar j}
-
T^j_{ik,\bar s}.
\]
Moreover,
\[
\sum_{j,s}\left|T^s_{ik,\bar j}\right|^2
=
\sum_{j,s}\left|T^j_{ik,\bar s}\right|^2.
\]
Hence
\begin{align*}
2\operatorname{Re}\sum_{j,s}
N_{ik\bar j\bar s}\overline{T^j_{ik,\bar s}}
&=
2\operatorname{Re}\sum_{j,s}
T^s_{ik,\bar j}\overline{T^j_{ik,\bar s}}
-
2\sum_{j,s}\left|T^j_{ik,\bar s}\right|^2\\
&=
-\sum_{j,s}
\left|
T^s_{ik,\bar j}
-
T^j_{ik,\bar s}
\right|^2\\
&=
-\sum_{j,s}
\left|N_{ik\bar j\bar s}\right|^2.
\end{align*}
Substituting into \eqref{eq:N-divergence-integration},
\[
\int_M\operatorname{Re}\ip{\DivCh N}{T}\,\dd V
=
\frac12\int_M|N|^2\,\dd V
=
\frac12\norm{N}_{L^2}^2.
\]
If \(H_g^c\equiv c\), then \eqref{eq:N-norm} gives
\[
\frac12\norm{N}_{L^2}^2
=
32\norm{\sigma}_{L^2}^2.
\]
\end{proof}

Combining the preceding identities gives two square formulas.

\begin{theorem}\label{thm:squares}
Let $(M^3,J,g)$ be a compact balanced threefold with
$H_g^c\equiv c$. Then
\begin{align}
\norm{\nabla^{0,1}T}_{L^2}^2
&=
32\norm{\sigma}_{L^2}^2
-c\norm{T}_{L^2}^2,
\label{eq:square-01}\\
\norm{\nabla^{1,0}T}_{L^2}^2
&=
9c\norm{T}_{L^2}^2
-16\norm{\sigma}_{L^2}^2.
\label{eq:square-10}
\end{align}
\end{theorem}

\begin{proof}
In a local unitary frame, set
\[
A
=
\sum_s
\overline{\varphi^s}\otimes\nabla_{\bar e_s}T.
\]
Applying Lemma~\ref{lem:balanced-ibp} with $B=T$, and using the
definition of $\mathcal U$, gives
\begin{equation}\label{eq:laplacian-01-integration}
\begin{aligned}
\int_M\operatorname{Re}\ip{\mathcal U}{T}\,\dd V
&=
\int_M\operatorname{Re}
\sum_s
\ip{\nabla^2_{e_s,\bar e_s}T}{T}\,\dd V\\
&=
-\int_M
\sum_s\left|\nabla_{\bar e_s}T\right|^2\,\dd V\\
&=
-\norm{\nabla^{0,1}T}_{L^2}^2.
\end{aligned}
\end{equation}
Similarly, set
\[
A
=
\sum_s
\varphi^s\otimes\nabla_{e_s}T.
\]
The conjugate identity in Lemma~\ref{lem:balanced-ibp} gives
\begin{equation}\label{eq:laplacian-10-integration}
\int_M
\operatorname{Re}\ip{\mathcal V}{T}\,\dd V
=
-\norm{\nabla^{1,0}T}_{L^2}^2.
\end{equation}
Subtracting \eqref{eq:laplacian-10-integration} from
\eqref{eq:laplacian-01-integration}, we obtain
\begin{equation}\label{eq:laplacian-difference-integration}
\int_M
\operatorname{Re}\ip{\mathcal U-\mathcal V}{T}\,\dd V
=
\norm{\nabla^{1,0}T}_{L^2}^2
-
\norm{\nabla^{0,1}T}_{L^2}^2.
\end{equation}

By \eqref{eq:holomorphic-laplacian-div},
\eqref{eq:curvature-action}, \eqref{eq:N-divergence-general}, and
\eqref{eq:N-divergence},
\begin{align*}
-\norm{\nabla^{0,1}T}_{L^2}^2
&=
\int_M
\operatorname{Re}\ip{\mathcal U}{T}\,\dd V\\
&=
\int_M
\operatorname{Re}
\ip{\mathcal R(T)-\DivCh N}{T}\,\dd V\\
&=
\int_M
\operatorname{Re}\ip{\mathcal R(T)}{T}\,\dd V
-
\int_M
\operatorname{Re}\ip{\DivCh N}{T}\,\dd V\\
&=
c\norm{T}_{L^2}^2
-\frac12\norm{N}_{L^2}^2\\
&=
c\norm{T}_{L^2}^2
-32\norm{\sigma}_{L^2}^2.
\end{align*}
This proves \eqref{eq:square-01}.

Next, \eqref{eq:laplacian-difference-integration},
\eqref{eq:laplacian-contraction}, and \eqref{eq:sigma-Qt} give
\begin{align*}
\norm{\nabla^{1,0}T}_{L^2}^2
-
\norm{\nabla^{0,1}T}_{L^2}^2
&=
\int_M
\operatorname{Re}\ip{\mathcal U-\mathcal V}{T}\,\dd V\\
&=
-2c\norm{T}_{L^2}^2
+
12\int_M\ip{\sigma}{\mathcal Q}\,\dd V\\
&=
-2c\norm{T}_{L^2}^2
+
12\bigl(
c\norm{T}_{L^2}^2
-
4\norm{\sigma}_{L^2}^2
\bigr)\\
&=
10c\norm{T}_{L^2}^2
-
48\norm{\sigma}_{L^2}^2.
\end{align*}
Combining this with \eqref{eq:square-01}, we obtain
\begin{align*}
\norm{\nabla^{1,0}T}_{L^2}^2
&=
32\norm{\sigma}_{L^2}^2
-c\norm{T}_{L^2}^2
+10c\norm{T}_{L^2}^2
-48\norm{\sigma}_{L^2}^2\\
&=
9c\norm{T}_{L^2}^2
-16\norm{\sigma}_{L^2}^2.
\end{align*}
This proves \eqref{eq:square-10}.
\end{proof}

\subsection{Proof of the balanced theorem}\label{subsec:balanced-proof}

\begin{proof}[Proof of Theorem~\ref{thm:balanced-main}]
Suppose first that $c=0$.  Equation \eqref{eq:square-10} becomes
\[
  \norm{\nabla^{1,0}T}_{L^2}^2
  =-16\norm{\sigma}_{L^2}^2.
\]
Since the left-hand side is nonnegative and the right-hand side is
nonpositive, both sides vanish.  Thus \(\sigma=0\) and
\(\nabla^{1,0}T=0\).  Equation \eqref{eq:square-01} gives
\(\nabla^{0,1}T=0\), and Proposition~\ref{prop:N-reconstruction} gives
\(N=0\).  Substitution in \eqref{eq:full-reconstruction} gives
\(R^c=0\).  This does not imply \(T=0\); the metric is Chern flat but
need not be K\"ahler.

Now suppose that $c<0$.  From \eqref{eq:square-10},
\[
  \norm{\nabla^{1,0}T}_{L^2}^2
  =9c\norm{T}_{L^2}^2-16\norm{\sigma}_{L^2}^2.
\]
The left-hand side is nonnegative and both terms on the right-hand side
are nonpositive.  Hence \(T=0\) and \(\sigma=0\).  The vanishing of the
Chern torsion makes \(g\) K\"ahler.  Proposition~\ref{prop:N-reconstruction}
also gives \(N=0\), and \eqref{eq:full-reconstruction} becomes
\[
 R^c_{i\bar j k\bar\ell}
 =\frac c2\bigl(
   \delta_{ij}\delta_{k\ell}+\delta_{i\ell}\delta_{kj}
 \bigr).
\]
\end{proof}

The same identities also show what this method yields in the positive
case.

\begin{corollary}\label{cor:positive-bounds}
Let \((M^3,J,g)\) be a compact balanced threefold.  If
\(H_g^c\equiv c>0\), then
\begin{equation}\label{eq:positive-bounds}
  \frac{c}{32}\norm{T}_{L^2}^2
  \le \norm{\sigma}_{L^2}^2
  \le \frac{9c}{16}\norm{T}_{L^2}^2.
\end{equation}
In particular, these square identities alone do not exclude a non-K\"ahler metric when $c>0$.
\end{corollary}

\begin{proof}
The two inequalities follow from the nonnegativity of the left-hand sides of \eqref{eq:square-01} and \eqref{eq:square-10}.
\end{proof}

\noindent\textbf{Acknowledgments.} The authors would like to express their sincere gratitude to Professors Xiaolan Nie and Fangyang Zheng for their insightful guidance, valuable discussions. They also thank Yijian Zhang and Shiyu Zhang for helpful
discussions and valuable suggestions concerning the balanced case.

%

\end{document}